\documentclass{amsart}
\usepackage{amsmath}
\usepackage{amssymb}
\usepackage[T2A,T1]{fontenc}
\usepackage[dvipsnames]{xcolor}
\usepackage{amsthm}
\usepackage{thmtools}
\usepackage{thm-restate}
\usepackage{amsfonts}
\usepackage{graphicx}
\usepackage{stackrel}
\usepackage[style=alphabetic, backend = biber]{biblatex}
\usepackage{wrapfig}
\usepackage{pgfplots}
\usepackage{tikz}
\usepackage{tikz-cd}
\usepackage{upgreek}
\usepackage{caption}
\usepackage{subcaption}
\usepackage{relsize}
\usepackage{mathtools}
\usepackage{adjustbox}
\usepackage{tabularx}
\usepackage[shortlabels]{enumitem}
\usepackage[hyperindex,colorlinks,breaklinks]{hyperref}
\usepackage{hyperref}
\usepackage[hyphenbreaks]{breakurl}
\usetikzlibrary{decorations.markings}
\usepackage{mathdots}

\newcommand{\G}{\Gamma}

\newcommand{\Gh}{\widehat{\G}}

\newcommand{\Gs}{\mathcal{G}}
\newcommand{\Gsc}{\overline{\mathcal{G}}}

\renewcommand{\L}{\Lambda}
\newcommand{\Lh}{\widehat{\L}}

\newcommand{\Chi}{\raisebox{0.175em}{$\upchi$}}

\newcommand{\ggx}{(\Gs,\Xi)}

\newcommand{\gxt}{(\Gs,\Xi,\Theta)}
\newcommand{\ggc}{(\Gsc,\Xi)}

\newcommand{\PP}[1][\ggx]{\Pi_1{#1}}
\newcommand{\PA}[1][\ggx]{\pi_1 #1}

\newcommand{\Cc}{\mathcal{C}}
\newcommand{\Sc}{\mathcal{S}}
\newcommand{\Ic}{\mathcal{I}}
\newcommand{\Gc}{\G_{\widehat{\Cc}}}
\newcommand{\Gsp}{\G_{\widehat{\Sc}}}
\newcommand{\fh}{\widehat{f}}
\newcommand{\yi}{\mbox{\usefont{T2A}{\rmdefault}{m}{n}\cyryi}}

\newcommand{\Z}{\mathbb{Z}}
\newcommand{\Fp}{\mathbb{F}_p}
\newcommand{\pZ}{\widehat{\Z}}
\newcommand{\at}[1]{|_{#1}}

\newcommand{\Aut}{\operatorname{Aut}}
\newcommand{\cd}{\operatorname{cd}}

\newtheorem{theorem}{Theorem}
\newtheorem{definition}{Definition}
\numberwithin{theorem}{section}
\numberwithin{definition}{section}
\numberwithin{equation}{section}
\newtheorem{lemma}[theorem]{Lemma}
\newtheorem{example}{Example}
\numberwithin{example}{section}
\newtheorem{question}{Question}
\numberwithin{question}{section}
\newtheorem{proposition}[theorem]{Proposition}

\newtheorem{corollary}[theorem]{Corollary}

\title[Cohomological Separability of GBS Groups]{Cohomological Separability of Baumslag--Solitar Groups and Their Generalisations}
\author{William D. Cohen and Julian Wykowski}
\date{\today}
\address{Department of Pure Mathematics and Mathematical Statistics, Centre for Mathematical Sciences, Wilberforce Road, Cambridge CB3 0WA}
\email{wdc26@cam.ac.uk}
\email{jw2006@cam.ac.uk}

\definecolor{imperiallight}{RGB}{24,142,179}
\definecolor{grn}{RGB}{68, 179, 24}
\definecolor{purple}{RGB}{129,39,137}
\hypersetup{
    breaklinks=true,
	linkcolor={grn},
	citecolor={grn},
	urlcolor={grn}
}
\usetikzlibrary{ decorations.markings}

\tikzcdset{
  cells={font=\everymath\expandafter{\the\everymath\displaystyle}},
}

\begin{document}
\begin{abstract}
    A group $\G$ has \emph{separable cohomology} if the profinite completion map $\iota \colon \G \to \widehat{\G}$ induces an isomorphism on cohomology with finite coefficient modules. In this article, cohomological separability is decided within the class of generalised Baumslag--Solitar groups, i.e. graphs of groups with infinite cyclic fibers. Equivalent conditions are given both explicitly in terms of the defining graph of groups and in terms of the induced topology on vertex groups. Restricted to the class of Baumslag--Solitar groups, we obtain a trichotomy of cohomological separability and cohomological dimension of the profinite completions. In particular, 
    this yields examples of non-residually-finite one-relator groups which have separable cohomology, and examples which do not.
\end{abstract}
\maketitle
\section{Introduction}
The question of profinite rigidity asks to what extent properties of an abstract group $\G$ are visible in its finite quotients. This problem can be rephrased as the question to what extent the properties of $\G$ can be detected in its profinite completion $\Gh$, the latter being a profinite group given by the inverse limit of the inverse system of the finite quotients of $\G$ and their epimorphisms. We invite the reader to \cite{Reid_Survey, Bridson_Survey} for a survey of major questions and results regarding the subject.

One such question is that of \emph{cohomological separability}, i.e. whether the profinite completion map $\iota \colon \G \to \Gh$ of an abstract group $\G$ induces an isomorphism on group cohomology with finite coefficient modules. This concept was first introduced under the name of \emph{cohomological goodness} as an exercise in \cite[Section I.2.6]{Serre_Cohomology} but has recently resurfaced as a question of research in its own right \cite{GJZ, Lorensen08, Kropholler_Wilkes, Wilton_Zalesskii}. Relating the cohomology of $\G$ to that of $\Gh$ is especially useful because it allows for the translation of difficult questions in the non-commutative realm of group theory to more approachable questions in the realm of commutative (or even linear) algebra. Recent applications range from group theory to algebraic geometry \cite{Schroer, BCR, Wilkes_SFS, Jaikin_Fiber, Wykowski}.

Typically, one restricts the study of profinite rigidity to abstract groups $\G$ which are finitely generated and residually finite. The latter assumption is necessary to ensure the weakest form of visibility of $\G$ in $\Gh$: that every non-zero element of $\G$ survives in some finite quotient thereof, or equivalently, that the profinite completion map $\iota \colon \G \to \Gh$ is an injection. The assumption of residual finiteness simplifies matters by allowing for identification of $\G$ and its subgroups with their images in $\Gh$ embedded as abstract subgroups. The question of which groups are residually finite is an active area of research itself (see \cite{Wilkes_Book}, for example).

However, there is no fundamental reason why one should consider profinite invariants only among residually finite groups. In this article, we consider the question of cohomological separability for \emph{Baumslag--Solitar groups}, a natural class of abstract groups most of which are not residually finite. Given a pair of integers $n,m \in \Z$, one defines the associated Baumslag--Solitar group as the group given by the presentation
\[
\operatorname{BS}(n,m) = \langle a,t \mid a^n = ta^mt^{-1} \rangle
\]
i.e., an HNN extension of the infinite cyclic group $\Z$ along an isomorphic subgroup included with indices $|n|$ and $|m|$, respectively. These groups are residually finite if and only if $|n| = |m|$ or one of $|n|,|m|$ is equal to one \cite[Theorem C]{Meskin}. We present the following trichotomy, which fully characterises cohomological separability and profinite cohomological dimension for Baumslag--Solitar groups in terms of their defining pairs of integers, and where we say that two numbers $n,m$ are \emph{isocratic} if for any prime $p$, either $p$ divides at most one of $n,m$, or it divides both with equal power (cf. Definition~\ref{Def::Iso}):
\begin{restatable}{theorem}{MTA}\label{Thm::MTA}
    Let $n,m \in \Z$ be integers and $\G = \operatorname{BS}(n,m)$ be the associated Baumslag--Solitar group. Exactly one of the following cases holds:
    \begin{enumerate}
        \item the numbers $n,m$ are either coprime or satisfy $|m|=|n|$, in which case $\cd(\G) = \cd(\Gh) = 2$ and $\G$ has separable cohomology;
        \item the numbers $n,m$ are isocratic but not coprime and $|n| \neq |m|$, in which case $\cd(\G) = \cd(\Gh) = 2$ but $\G$ does not have separable cohomology; and
        \item the numbers $n,m$ are not isocratic, in which case $\cd(\G) = 2$ but $\Gh$ has torsion, so $\cd(\Gh) = \infty$ and $\G$ does not have separable cohomology.
    \end{enumerate}
\end{restatable}
Notably, the first case includes all residually finite Baumslag--Solitar groups. The proof utilises the decomposition of a Baumslag--Solitar group $\G$ as a graph of groups whose vertex groups have separable cohomology. A group with such a decomposition is known to have separable cohomology, provided it is efficient (c.f. Section~\ref{sec:prelimSep}): see \cite[Proposition 3.6]{GJZ}. However, in the case where the graph of groups is not necessarily efficient, one may still relate the cohomology of $\G$ to that of $\Gh$, provided one is prepared to work with subgroups $\Delta \leq \G$ whose images in $\Gh$ might be proper quotients (if $\G$ is not residually finite) and whose closures $\overline{\iota\Delta}$ might be proper quotients of their profinite completions $\widehat{\Delta}$ (if $\G$ does not induce the full profinite topology on $\Delta$). The induced graph of profinite groups and diagram of Mayer--Vietoris sequences is established in Proposition \ref{Prop::ComDiag}. The added cost is that cohomological calculations are now more complex; these are reported on in Section~\ref{Sec::BS}. See Theorem~\ref{Thm::MTALoopCase} for the proof of (a generalised version of) Theorem~\ref{Thm::MTA}.

Noting that Baumslag--Solitar groups are one-relator groups, the above theorem yields first examples of non-residually-finite one-relator groups with separable cohomology, and examples that do not have separable cohomology.
\begin{example}
    Let $p,q \in \Z_{\geq 1}$ be coprime integers greater than 1. The Baumslag--Solitar group $\operatorname{BS}(p,q)$ is a one-relator group which is not residually finite but has separable cohomology.
\end{example}
\begin{example}
    The Baumslag--Solitar group $\operatorname{BS}(2,4)$ is a finitely generated one-relator group which does not have separable cohomology.
\end{example}
We invite the reader to compare the above examples to the following open question which has been repeatedly asked at research conferences.
\begin{question}[Folklore]\label{Quest}
    Does any finitely generated residually finite one-relator group have seaparable cohomology?
\end{question}
A natural generalisation of Baumslag--Solitar groups is the class of fundamental groups of graphs of groups with infinite cyclic vertex and edge groups. These are commonly referred to as \emph{generalised Baumslag--Solitar (GBS) groups} and have been studied in \cite{Kropholler90, Levitt07, Levitt15}. These groups are determined by the inclusion indices of edge into vertex groups, and are hence described by finite graphs with edges labelled by pairs of integers. As a generalisation of Theorem~\ref{Thm::MTA}, we decide the question of cohomological separability within this class of GBS groups. The following result characterises this property in terms of certain products of inclusion indices; we refer the reader to Section \ref{Sec::Prelim} for definitions. 
\begin{restatable}{theorem}{MTB}
\label{Thm::MTB}
    Let $\G = \PA$ be a generalised Baumslag--Solitar group over a reduced graph of groups $\ggx$. Then $\G$ has separable cohomology if and only if one of the following conditions holds:
    \begin{enumerate}
        \item The graph $\Xi$ is a cycle and the augmentation products $n(\Xi)$ and $m(\Xi)$ are coprime; or
        \item the following equivalent conditions hold:
        \begin{enumerate}
            \item for every cycle $\Upsilon \subseteq \Xi$, the equation $n(\Upsilon) = m(\Upsilon)$ holds, or
            \item the group $\G$ induces the full profinite topology on each vertex group.
        \end{enumerate}
    \end{enumerate}
\end{restatable}
We refer the reader to Section~\ref{Sec::GBS} for the proof. The paper is structured as follows. In Section~\ref{Sec::Prelim}, we outline preliminary results regarding the cohomology of profinite groups, graphs of groups decompositions and GBS groups. In Section~\ref{Sec::Gen}, we exhibit general results relating the cohomology of a graph of groups $\G$ and its profinite completion $\Gh$, as well as direct corollaries concerning GBS groups. In Section~\ref{Sec::BS}, we consider GBS groups over a cyclic graph, culminating in a proof of Theorem~\ref{Thm::MTA}. Finally, in Section~\ref{Sec::GBS}, we generalise these results to the class of all GBS groups, concluding with a proof of Theorem~\ref{Thm::MTB}.
\section*{Acknowledgements}
The authors would like to thank their respective PhD supervisors, Jack Button and Gareth Wilkes, for plentiful valuable discussions and insights throughout the project. The authors would also like to thank Francesco Fournier-Facio and {Pawe\l} Piwek for their helpful comments. Financially, the first author is grateful for the support of the Basil Howard Research Graduate Scholarship, while the second author is grateful for support of the Cambridge International Trust and King’s College (Harold Fry Fund) Scholarship.
\section{Preliminaries}
\label{Sec::Prelim}
\subsection{Cohomology of Profinite Groups}
We first note that in this article all modules will be left modules unless stated otherwise.
Let $G$ be a profinite group. A \emph{$G$-module} is an abelian topological group $M$ equipped with a continuous $G$-action. We define the cohomology groups of $G$ with coefficients in $M$ as the direct limit
\[
H^\bullet(G,M) = \lim_{U \longrightarrow} H^\bullet(G/U,M^U)
\]
where $U \trianglelefteq_o G$ runs over all open normal subgroups of $G$. Note that given an abstract group $\Gamma$ with profinite completion map $\iota \colon \Gamma \to \Gh$, any $\Gh$-module $M$ forms a $\Gamma$-module with action via $\iota$. Conversely, if $M$ is a \emph{finite} $\Gamma$-module, the action $\Gamma \to \operatorname{Aut}(M)$ factors through a map $\Gh \to \operatorname{Aut}(M)$, making $M$ a discrete $\Gh$-module.

First cohomology groups will have special significance in this article. The first cohomology group of an abstract or profinite group is given by the set of its derivations modulo its inner derivations; see \cite[Section IV.2]{brown_cohomology} and \cite[Section 6.8]{Bible} for the abstract and profinite versions, respectively. Recall that a \emph{(left) derivation} or \emph{crossed homomorphism} of a group $\Gamma$ in a (left) $\Gamma$-module $M$ is a map $f \colon \Gamma \to M$ satisfying $f(gh) = f(g) + g f(h)$ for all $g,h \in \Gamma$. In the special case where $\G\cong \Z=\langle \omega\rangle$, we thus have that for all $\G$-modules $M$ the group $H^1(\G, M)\cong M/(\omega-1)M$ along the map $[f]\mapsto[f(\omega)]$.

While the above definition of cohomology is valid for any topological $G$-module $M$, the theory quickly develops pathologies if one considers all such modules. For this reason, one usually restricts to the category $\mathbf{DT}(G)$ of discrete torsion modules, which has enough injectives and where an alternative functional definition via homological algebra exists. We then define the \emph{cohomological dimension} of a profinite group $G$ as
\[
\cd(G) = \sup\{n \in \Z_{\geq 0} \mid \exists M \in \mathbf{DT}(G), H^n(G,M) \neq 0\}
\]
and we write $\cd(G) = \infty$ if the supremum does not exist. We remark that one may define a functional cohomlogy theory also for coefficients in the category of profinite modules, as long as the profinite group satisfies a certain finiteness condition. While this will not be necessary in the present article, we refer the curious reader to \cite{Symonds_Weigel} for details.

The following lemma permits us to separate cohomologies with coefficients in finite modules by the prime decomposition of its order. This will be useful for a characterisation of cohomological separability of (generalised) Baumslag--Solitar groups based on the prime decompositions of their inclusion indices.
\begin{lemma}\label{Lem::CoprimeCohomologyVanishes}
    Let $\Sc$ be a set of primes and $G$ a pro-$\Sc$ group. If $M$ is a finite $G$-module whose order is coprime to $\Sc$ then
    \[
        H^\bullet(G,M) = 0
    \]
    holds.
\end{lemma}
\begin{proof}
    Recall that the cohomology of a profinite group $G$ with coefficients in a profinite module $M$ can be decomposed as a direct limit
    \[
    H^n(G,M) = \lim_{ U\rightarrow} H^n(G/U,M^U)
    \]
    where $n \in \Z_{\geq 0}$ and $U \trianglelefteq_o G$ ranges through the inverse system of open normal subgroups of $G$. Thus it suffices to show the result for finite groups $G$. In that case, the order of any element $[\sigma] \in H^n(G,M)$ divides the order of $G$: see e.g. \cite[Corollary 6.7.4]{Bible}. On the other hand, the order of $\sigma$ must also divide the order of $M$, as multiplication by $|M|$ annihilates all elements of $M$. But $|G|$ is an $\Sc$-group by assumption, so the order of $[\sigma]$ as a group element divides $\gcd(|G|,|M|) = 1$ and $H^n(G,M)$ must be trivial.
\end{proof}
\subsection{Abstract and Profinite Graphs of Groups}
We will assume the reader has some familiarity with Bass-Serre theory, and for a more detailed discussion we refer to \cite{Serre, dicks2011groups}. Let $\Xi$ be a graph. For an edge $e\in E(\Xi)$ we denote by $d_0(e)$ and $d_1(e)$ the initial and terminal vertices of $e$ respectively.

A \emph{graph of groups} $(\Gs, \Xi)$ over a finite connected graph $\Xi$ consists of:
\begin{enumerate}[(a)]
    \item a collection of abstract groups $\{\Gs(x) : x \in \Xi\}$, referred to as the \emph{vertex} and \emph{edge groups}, or collectively as the \emph{fibres}; and
    \item two collections of monomorphisms of groups $\{\partial_{e, 0} \colon \Gs(e) \to \Gs(d_0(e))\}_{e \in E(\Xi)}$ and $\{\partial_{e, 1} \colon \Gs(e) \to \Gs(d_1(e))\}_{e \in E(\Xi)}$, referred to as the \textit{edge inclusions}.
\end{enumerate}
Let $(\Gs, \Xi)$ be a graph of groups. We define the fundamental group $\PA$ of $(\Gs, \Xi)$ as in \cite[Section~I.5.1]{Serre}.

Now, we outline the analogous definitions in the profinite category. A \textit{profinite graph} $X$ is a compact, Hausdorff and totally disconnected space equipped with a distinguished closed subspace $V(X) \subseteq X$, called the \emph{vertex set}, as well as two continuous maps $d_0, d_1 \colon X \to V(X)$ satisfying $d_0\at{V(X)} = d_1\at{V(X)} = \operatorname{id}_{V(X)}$.
A \emph{graph of profinite groups} $\ggx$ over a finite connected graph $\Xi$ consists of:
\begin{enumerate}[(a)]
    \item a collection of profinite groups $\{\Gs(x) : x \in \Xi\}$, referred to as the \emph{fibres}; and
    \item two collections of continuous injective homomorphisms of profinite groups $\{\partial_{e, 0} \colon \Gs(e) \to \Gs(d_0(e))\}_{e \in E(\Xi)}$ and $\{\partial_{e, 1} \colon \Gs(e) \to \Gs(d_1(e))\}_{e \in E(\Xi)}$, referred to as the \textit{edge inclusions}.
\end{enumerate} 
Given a graph of profinite groups $\ggx$, we want to associate to it a fundamental group as in the abstract case, which will be a group generated by the vertex groups and a collection of stable letters $T = \{t_e : e \in E(\Xi)\}$, indexed by the edges of $\Xi$, which are glued together with respect to the structure of $\Xi$. Specifically, choose a spanning tree $\Theta$ of $\Xi$ and consider the profinite free product
\[
W \ggx =  \coprod_{v \in V(\Xi)} \Gs(\tau) \amalg \pZ[[T]]
\]
where $\pZ[[T]]$ denotes the free profinite group on the space $T$. Let $N \gxt \trianglelefteq W \ggx$ be the minimal closed normal subgroup containing the sets $\{t_e : e \in E(\Theta)\}$ and $\{\partial_1^{-1}(g)t_e^{-1}\partial_0(g)t_e : e \in E(\Xi), g \in \Gs(e)\}$. The \emph{profinite fundamental group of $\ggx$ with respect to $\Theta$} is defined as
\[
\PP[(\Gs, \Xi, \Theta)] = \frac{W\ggx}{N\gxt}
\]
which is profinite as $N\ggx$ is closed in $W\ggx$. Given a different choice of spanning tree $\Theta'$ of $\Xi$, there is an isomorphism
\begin{equation}
	\PP[(\Gs, \Xi, \Theta)] \cong \PP[(\Gs, \Xi, \Theta')]
\end{equation}
although this isomorphism may carry vertex groups to distinct conjugates. Thus, we shall write $\PP$ whenever we refer to the isomorphism type of this group, and specify a spanning tree only when we need to refer to the images of vertex groups in the fundamental group. Moreover, to distinguish this construction from the abstract case, we shall write $\PA$ and $\PP$ to denote the abstract and profinite fundamental groups of a graph of groups $\ggx$, respectively. Unlike in the abstract setting, the canonical morphisms $\Gs(x) \to \PP$ may not be injective; if they are indeed monomorphisms, we say that $\gxt$ is an \emph{injective} graph of profinite groups. We note that a graph of profinite groups which is not injective can be replaced with a natural construction which does form an injective graph of profinite groups: see \cite[Section 6.4]{Ribes_Graphs}.

A \emph{generalised Baumslag--Solitar (GBS) group} is the fundamental group of a graph of groups $\G = \PA$ over a finite graph $\Xi$, whose vertex and edge groups are all isomorphic to the infinite cyclic group $\Z$. In this case, we fix isomorphisms $\vartheta_x \colon \Gs(x) \to \Z$ for $x \in \Xi$, hereinafter referred to as the ``canonical isomorphisms''. As inclusions $\Z \hookrightarrow \Z$ are in bijective correspondence with $\Z$, a GBS group over a graph $\Xi$ is determined by a function assigning edges to pairs of vertices (one each for the initial and terminal vertex inclusions). For the purpose of separable cohomology, it suffices to work with the images of these inclusions, for which reason we may drop the sign and define the \emph{inclusion indices} as the functions
\[
\yi_0, \yi_1 \colon E(\Xi) \to \Z
\]
given by
\[\yi_j(e) = [\Gs(d_j(e)) : \Gs(e)]\] for $e \in E(\Xi)$ and $j \in \{0,1\}$. The characterisation of separable cohomology of GBS groups will then be given in terms of the products of the $\yi$-values along a certain subgraphs. For this reason, we define
\[
n(U) = \prod_{e\in U} \yi_0(e) \quad \text{and} \quad m(U) = \prod_{e\in U} \yi_1(e)
\]
whereto we refer as the \emph{augmentation products} of a subset $U \subseteq E(\Xi)$. If $\Xi$ is a cycle we will always calculate $m(\Xi)$ and $n(\Xi)$ as though all edges of $\Xi$ are oriented as in Figure~\ref{Fig:HEXAGON}.

Let $\G$ be a generalised Baumslag--Solitar group over a graph $\Xi$, and let $\Theta$ be the intersection of every spanning tree of $\Xi$. If there exists some edge $e\in E(\Theta)$ such that either $\yi_0(e)=1$ or $\yi_1(e)=1$ then we may collapse $e$ in $\Xi$ without changing the isomorphism type of $\G$. We say that a generalised Baumslag--Solitar group over a graph with no such edges is over a \emph{reduced} graph. 

Similarly, we may add a new vertex $v'$ to $\Xi$ in the middle of any edge with both monomorphisms into $\Gs(v')$ given by the identity map. In this way we may assume that every edge in $\Xi$ has two distinct endpoints whenever convenient, but this construction will not result in a reduced graph of groups in general.

If $\Xi$ is a cycle graph then we say that $\G$ is a generalised Baumslag--Solitar (GBS) group \emph{over a cycle}. In this case, we shall label the graph $\Xi$ as $V(\Xi) = \{v_1, \ldots, v_s\}$ satisfying $d_0(e_i) = v_i$ and $d_1(e_i) = v_{i+1}$ for $1 \leq i < s$. While irrelevant to the statement of the result, during calculations it will occasionally be convenient to refer to the image of the generator under the inclusion map with the proper sign. As such, we shall write $n_i$ and $m_i$ for the images of the generator $1\in \Z$ under the maps
\[\partial_{e_i,j}\circ\vartheta^{-1}_{e_i}\colon\Z\rightarrow \Gs(d_j(e_i))\]
for $j=0$ and $j=1$ respectively. This situation is illustrated in Figure~\ref{Fig:HEXAGON}.
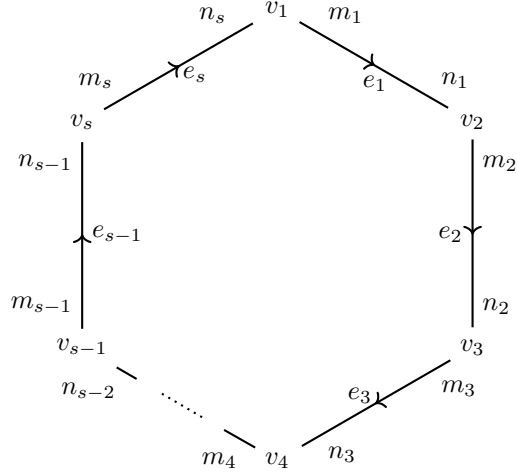
\begin{figure}
    \def\scale{1.5}
        \centering
    \begin{tikzpicture}[->,shorten >=1pt,auto,node distance=3cm,
            thick]

        \node (a1) at (0, 0) {$v_1$};
        \node (a2) at (1.73*\scale, -1*\scale) {$v_2$};
        \node (a3) at (1.73*\scale, -3*\scale) {$v_3$};
        \node (a4) at (0, -4*\scale) {$v_4$};
        \node (as-1) at (-1.73*\scale, -3*\scale) {$v_{s-1}$};
        \node (as) at (-1.73*\scale, -1*\scale) {$v_{s}$};

        \node (mid1) at ( -1.16*\scale,-3.33*\scale) {};
        \node (mid2) at ( -0.58*\scale,-3.67*\scale) {};
         \begin{scope}[decoration={
                        markings,
                        mark=at position 0.5 with {\arrow{>}}}
                      ]    
            \draw[-, postaction={decorate}] (a1) -- (a2)
                node[midway, below] {$e_1$}
                node[very near start] {$m_1$}
                node[very near end] {$n_1$};
            \draw[-, postaction={decorate}] (a2) -- (a3)
                node[midway, left] {$e_2$}
                node[very near start] {$m_2$}
                node[very near end] {$n_2$};
            \draw[-, postaction={decorate}] (a3) -- (a4)
                node[pos=0.5+0.15/\scale, above] {$e_3$}
                node[very near start] {$m_3$}
                node[very near end] {$n_3$};
            \draw[-] (a4) -- (mid2)
                node[near start] {$m_4$};
            \draw[-, dotted]  (mid2) -- (mid1) {};
            \draw[-] (mid1) -- (as-1)
                node[midway] {$n_{s-2}$};
            \draw[-, postaction={decorate}] (as-1) -- (as)
                node[midway, right] {$e_{s-1}$}
                node[very near start] {$m_{s-1}$}
                node[very near end] {$n_{s-1}$}; 
            \draw[-, postaction={decorate}] (as) -- (a1)
                node[pos=0.5+0.15/\scale, below] {$e_{s}$}
                node[very near start] {$m_{s}$}
                node[very near end] {$n_{s}$};
        \end{scope}
    \end{tikzpicture}
    \caption{\label{Fig:HEXAGON}An illustration of the standard layout of a GBS group over graph with a single cycle.}
\end{figure}

Choosing the spanning tree $\Theta = \Xi - \{e_s\}$ for $\Xi$ and labelling the generator of $\Gs(v_i)$ as $a_i$, we obtain the canonical presentation
\begin{equation}\label{Eq::CanonicalPresentation}
    \G \cong \frac{\big\langle a_1, \ldots, a_s, t \big\rangle}{\big \langle \big \langle \{a_i^{n(e_i)}a_{i-1}^{-m(e_i)} : i\neq s\} \cup \{ta_s^{n_s}t^{-1}a_1^{-m_s}\} \big \rangle \big \rangle}
\end{equation}
for $\Gamma$, where $t$ is the stable letter of the HNN extension corresponding to the cycle.

Given an integer $x$ and a prime number $p$, write $\nu_p(x)$ for the $p$-adic valuation $\nu_p(x) = \max\{k \in \mathbb{Z} \mid p^k \text{ divides } x \}$. For the characterisation of the cohomological dimension of the profinite completion of a GBS group, we shall consider differences in the $p$-adic valuation of primes dividing the augmentation products. We make the following definition.
\begin{definition}\label{Def::Iso}
    Two integers $n,m$ are \emph{isocratic} if for every prime number $p$,
    \[
        \nu_p(n) \cdot \nu_p(m) \neq 0 \Rightarrow \nu_p(n) =\nu_p(m)
    \]
i.e., either $p$ divides at most one of $n$ or $m$, or it divides both with equal power.
\end{definition}
In Theorem~\ref{Thm::MTA}, the cohomological dimension of the profinite completion of a GBS group over a cycle is shown to be finite if and only if the augmentation products corresponding to the cycle are isocratic.

Now, let $\G = \PA$ be a generalised Baumslag--Solitar group over any graph $\Xi$. Given a subtree $\Theta$ of $\Xi$, a vertex $w\in \Theta$ and a prime number $p$, we may define the function $\varepsilon_{p, w} \colon V(\Theta) \to \Z$ as
\begin{equation}\label{Eq::Epsilon}
\varepsilon_{p, w}(v) = \begin{cases} \sum_{e\in E([v, w])} \nu_p(\yi_0(e)) - \nu_p(\yi_1(e)), & v\neq w \\
0 & v=w\end{cases}
\end{equation}
where for this calculation we reorient every edge of $\Theta$ away from $w$ and where $[v, w]$ is the unique geodesic joining $v$ to $w$ in $\Theta$. For a fixed spanning tree $\Theta$, the function $\varepsilon_{p, w}$ will depend on the choice of vertex $w$ only up to a constant, so we will omit $w$ from the notation whenever unnecessary.

However, the function $\varepsilon_{p}$ can vary wildly with our choice of spanning tree. For example, let $\Xi$ be a cycle. Pick some edge $e\in E(\Xi)$, and define $\Theta=\Xi-e$. The constants $n(\Xi),m(\Xi)$ are isocratic if and only if for every prime $p$ that divides both $n(\Xi)$ and $m(\Xi)$, the equation
\begin{equation}\label{eq:epsilonFormula}\varepsilon_{p}(v_{i})-\varepsilon_{p}(v_{i+i}) = \nu_p(\yi_1(e_i)) - \nu_p(\yi_0(e_i))\end{equation} holds for all $1\leq i\leq s$, where indices are added modulo $s$ and using for this calculation the standard orientation of edges clockwise around the cycle. Aditionally, (\ref{eq:epsilonFormula}) will always hold for a prime $p$ that does not divide the product $n(\Xi)m(\Xi)$, in which case $\varepsilon_p\equiv 0$. If $m(\Xi)$ and $n(\Xi)$ are isocratic and $p$ divides either both or neither of $n(\Xi)$ or $m(\Xi)$, the function $\varepsilon_p$ will depend on the choice of spanning tree only up to a constant. However, (\ref{eq:epsilonFormula}) can never hold for a prime $p$ that divides exactly one of $n(\Xi)$ and $m(\Xi)$ or divides them in different powers, and so in this case $\varepsilon_p$ will depend heavily on our choice of spanning tree. A simple example of such a case is the function $\varepsilon_2$ for the GBS group $\G$ over a $2$-cycle with $m_1= m_2= n_1=3$ and $n_2=2$.

The main use of the $\varepsilon_{p, w}$ function is to keep track of orders of images of generators of $\G$ in finite quotients. For example, let $\G$ be a GBS group over an isocratic cycle $\Xi$ with some prime $p$ dividing both $m(\Xi)$ and $n(\Xi)$, and $w\in V(\Xi)$. Assume that $\phi\colon \G\rightarrow Q$ is a finite quotient of $\Gamma$. If some vertex $v_i\in V(\Xi)$ minimises $\varepsilon_{p, w}$, and the order $|\phi(a_i)|$ of the image of $a_i$ in $Q$ is a power $p^k$ of $p$ such that $k> \varepsilon_{p, w}(v_j)-\varepsilon_{p, w}(v_i)$ for all $1\leq j\leq s$, then
\begin{equation}\label{Eq::Order}
    \nu_p(|\phi(a_j)|)=k+\varepsilon_{p, w}(v_i)-\varepsilon_{p, w}(v_j)
\end{equation}
for any vertex generator $a_j$ of $\G$ corresponding to a vertex $v_j$, as in (\ref{Eq::CanonicalPresentation}).

From the perspective of cohomology, the major advantage of a graph of groups decomposition $\G = \PA$ is the induced long exact sequence on cohomology groups, which allows one to relate the cohomology of $\G$ to the cohomology of fibers. This is known as the Mayer--Vietoris sequence, and was established (for abstract groups) in \cite[Theorem 2]{Chiswell}. We quote it below, with the addition of an explicit formula for one of the maps involved: this may be obtained by unpacking explicit forms for the involved chain maps and the Shapiro isomorphism. The corresponding statement in the profinite category is proven analogously (to appear as \cite[Theorem 8.4.6]{Wilkes_Book}, see also \cite[Theorem 9.4.1]{Ribes_Graphs} for the dual statement).
\begin{proposition}\label{Prop::Mayer-Vietoris}
    Let $\ggx$ be a discrete or profinite graph of groups over a finite graph $\Xi$ and $G$ its fundamental group, i.e. $G = \PA$ or $G = \PP$, respectively. For any discrete $G$-module $M$, there is a long exact sequence
    \[
    \ldots \to H^n(G,M) \xrightarrow{i^*} \bigoplus_{v \in V(\Xi)} H^n(\Gs(v),M) \xrightarrow{\hbar} \bigoplus_{e \in E(\Xi)} H^n(\Gs(e),M)  \xrightarrow{\delta} H^{n+1}(G,M) \to \ldots
    \]
    where $i^* \colon \bigoplus_{v \in V(\Xi)} H^n(\Gs(v),M) \to G$ is induced by the inclusions of vertex groups,
    \[\hbar\left((f_v)_{v \in V(\Xi)} \right) \colon \left((x_e)_{e \in E(\Xi)} \right) \mapsto \left(f_{d_1(e)}(\partial_{e, 1}(x_e)) - t_e \cdot f_{d_{e, 0}(e)}(\partial_0(x_e)) \right)\] and $\delta$ is a boundary homomorphism.
\end{proposition}
\begin{corollary} \label{Cor::CohomDim}
    Let $\G$ be the fundamental group of a graph of free groups. Then $\G$ has cohomological dimension at most $2$.
\end{corollary}
\subsection{Cohomological Separability}\label{sec:prelimSep}
Let $\G$ be a group and $\iota \colon \G \to \Gh$ be its profinite completion. We say that $\G$ has \emph{separable cohomology in dimension $n$} if the induced map
\[
\iota^* \colon H^n(\Gh, M) \to H^n(\G, M)
\]
is an isomorphism for all finite $\G$-modules $M$. We say that $\G$ has \emph{separable cohomology} if it has separable cohomology in all dimensions. This concept---more commonly known as \emph{cohomological goodness} or \emph{goodness in the sense of Serre}---was first introduced by Serre as an exercise in \cite[Section I.2.6]{Serre_Cohomology}. In the interest of descriptiveness, we opt instead for the terminology of ``cohomological separability'', first suggested by Gareth Wilkes in \cite{Wilkes_Book}.
We include the following two properties first observed by Serre in the aforementioned book, which shall be pertinent to the present article:
\begin{lemma}[\cite{Serre_Cohomology}]\label{Lem::DimOne}
    All groups have separable cohomology in dimension 1.
\end{lemma}
\begin{lemma}[\cite{Serre_Cohomology}]\label{Lem::Surjections}
    If $\iota^* \colon H^n(\Gh, M) \to H^n(\G, M)$ is an epimorphism for all finite $\G$-modules $M$ and $n \in \Z_{\geq 0}$ then it is an isomorphism for all finite $\G$-modules $M$ and $n \in \Z_{\geq 0}$, i.e. $\G$ has separable cohomology.
\end{lemma}

One of the main methods of proving cohomological separability in the literature consists of using an efficient graph of groups decomposition. Let $\ggx$ be a graph of groups with fundamental group $\G$. We say that this decomposition is \emph{efficient} if the following hold:
\begin{enumerate}[(E1)]
    \item The fundamental group $\G$ is residually finite;
    \item The group $\G$ induces the full profinite topology on all edge and vertex groups of $\ggx$; and
    \item All edge and vertex groups of $\ggx$ are closed in $\G$ under the profinite topology.
\end{enumerate}
It is an observation of Grunewald, Jaikin-Zapirain and Zalesskii \cite[Theorem~1.4]{GJZ} that the fundamental group of an efficient graph of groups whose fibres have separable cohomology must have separable cohomology itself, wherewith they show that Bianchi groups have separable cohomology. Going further, a result of Wilton and Zalesskii \cite{WiltonZalesskiiVirtSpec} proves that every virtually special group has separable cohomology using the fact that the hierarchy associated to the finite index special subgroup constructed by Haglund and Wise in \cite{HaglundWise} is efficient. More recently, Jankiewicz and Schreve have shown that an \emph{algebraically clean} graph of groups, i.e. a graph of free groups where all edge groups include as free factors of their neighbouring vertex groups, is efficient and thus has cohomologically separable fundamental group \cite[Theorem~1.2]{jankiewicz2023profinite}.

It is unknown in general when a one-relator group (even with the extra assumption of residual finiteness) has separable cohomology --- see Question~\ref{Quest} above. Some partial results are known, most notably that virtually special one-relator groups have separable cohomology as above, and many one-relator groups are known to be virtually special \cite[Theorems~8.1 and~8.6]{linton2022} (see also \cite{WiseBook, LouderWilton17}).

One large advantage of working with one-relator groups is that they come equipped with a natural hierarchical structure known as a \emph{Magnus hierarchy} (\cite{Magnus+1930+141+165}, see also \cite[Section~5]{linton2022}), which has been instrumental in proving many results, but which unfortunately often fails to be efficient. In the world of generalised Baumslag--Solitar groups we work instead with the natural graph of $\Z$s decomposition above, which also fails to be efficient but where we will be able to explicitly calculate the induced topologies on edge and vertex groups (see Proposition~\ref{Prop::Closure}).

\section{Profinite Cohomology of Graphs of Groups}\label{Sec::Gen}
In this section, we assemble general results regarding profinite completions of graphs of groups and their cohomologies, as well as some early corollaries regarding GBS groups. Given a graph of groups $\ggx$ with residually finite fundamental group $\G = \PA$, one may form a graph of profinite groups $\ggc$ whose fibers are the closures of the fibers of $\ggx$ in $\Gh$. Then there exists an isomorphism of profinite groups $\Gh \cong \PP[\ggc]$ which commutes with the respective inclusion maps \cite[Proposition 6.5.3]{Ribes_Graphs}.

However, one may dispose of the assumption of residual finiteness at the expense of having to work with fibers whose image in $\Gh$ might be a proper quotient. In the following proposition, we record this result and a cohomological consequence thereof. While the proof given in \cite[Proposition 6.5.3]{Ribes_Graphs} seems to generalise to the present setting without substantial adjustment, we provide instead an argument phrased in a slightly more categorical fashion.
\begin{proposition}\label{Prop::ComDiag}
    Let $\G = \PA$ be a graph of groups over a finite graph $\Xi$ and $\iota \colon \G \to \Gh$ its profinite completion. The inclusions of edge and vertex groups in $\G$ induce inclusions of the closures of their images in $\Gh$ which form a graph of groups $\Gsc(x) = \overline{\iota\Gs(x)}$ over $x \in \Xi$ whose profinite fundamental group is $\Gh = \PP[(\Gsc,\Xi)]$. Moreover, for any finite $\G$-module $M$, there is a commutative diagram 
    \[
\begin{tikzcd}
\vdots \arrow[d]                                               &&& \vdots \arrow[d]                                                       \\
{H^n(\Gh,M)} \arrow[d, "i^*"]    \arrow[rrr, "\iota^*"]                                &&& {H^n(\G,M)} \arrow[d, "i^*"]                    \\
{\bigoplus_{v \in V(\Xi)} H^n(\Gsc(v),M)} \arrow[d, "\hbar"] \arrow[rrr, "\mu_V = \prod_{v \in V(\Xi)} (\iota\at{\Gs(v)})^*"]    &&& {\bigoplus_{v \in V(\Xi)} H^n(\Gs(v),M)} \arrow[d, "\hbar"]   \\
{\bigoplus_{e \in E(\Xi)} H^n(\Gsc(e),M)} \arrow[d, "\delta"] \arrow[rrr, "\mu_E = \prod_{e \in E(\Xi)} (\iota\at{\Gs(e)})^*"] &&& {\bigoplus_{e \in E(\Xi)} H^n(\Gs(e),M)} \arrow[d, "\delta"] \\
{H^{n+1}(\Gh,M)} \arrow[d]\arrow[rrr, "\iota^*"]                                 &&& {H^{n+1}(\G,M)} \arrow[d]                     \\
\vdots                                                         &&& \vdots                                                                
\end{tikzcd}
    \]
    whose columns are the long exact sequences given in Proposition~\ref{Prop::Mayer-Vietoris}. 
\end{proposition}
\begin{proof}
    Given $x \in \Xi$, let $\Gsc(x) = \overline{\iota \Gs(x)}$ be the closure of the image of $\Gs(x)$ in the profinite completion of $\G$. The profinite group $\Gsc(e)$ is then the completion of $\Gs(e)$ with respect to the topology induced by the system of open subgroups of $\G$. It follows that the inclusion maps $\partial_{e,i} \colon \Gs(e) \to \Gs(d_i(e))$ extend to continuous monomorphisms $\partial_{e,i} \colon \Gsc(e) \to \Gsc(d_i(e))$ whenever $e \in E(\Xi)$ and $i=0,1$. This defines a graph of profinite groups $\ggc$ over the finite graph $\Xi$.
    
    Let $G = \PP[\ggc]$ be the fundamental group of this graph of profinite groups. The canonical maps $\varphi_v \colon \Gsc(v) \to G$ might not be monomorphisms in general (cf. \cite[Remark 6.2.6]{Ribes_Graphs}); nonetheless one may form the composition
    \[
    \psi_v \colon \Gs(v) \xrightarrow{ \iota} \Gsc(v) \xrightarrow{\varphi_v} G
    \]
    whenever $v \in V(\Xi)$. Together with the identity morphism on stable letters, the collection of maps $(\psi_v : v \in V(\Xi))$ induces a map on the fundamental group $\psi \colon \G \to G$, which factors through the profinite completion 
    \[
    \begin{tikzcd}
        & \Gh \arrow[rd, "\widehat \psi"] & \\
        \G \arrow[rr, "\psi"] \arrow[ru, "\iota"] && G
    \end{tikzcd}
    \]
    yielding a morphism of profinite groups $\widehat \psi \colon \Gh \to G$. As the stable letters and images of vertex groups generate $G$ topologically, the morphism $\widehat \psi$ must be surjective. For injectivity,
    choose $g \in \Gh - 1$ and let $U \trianglelefteq_o \G$ be an open normal subgroup for which $g \notin U$. The graph of finite groups $(\Gs_U,\Xi)$ given by $\Gs_U(x) = \frac{\Gs(x)}{\Gs(x) \cap U}$ over $x \in \Xi$ has a virtually free fundamental group $\PA[(\Gs_U,\Xi)]$ and the canonical projection
    \[
    \G \xrightarrow{\iota} \Gh \xrightarrow{p_U} \G / U
    \]
    factors as
    \[
    \G \xrightarrow{\varpi_U} \PA[(\Gs_U, \Xi)] \xrightarrow{\pi_U} \G / U
    \]
    where $\varpi_U$ is induced by the quotient on each vertex group and the identity on stable letters, while $\pi_U$ is induced by the quotient of the inclusions of vertex groups and stable letters. Moreover, the profinite fundamental group of the graph of finite groups $(\Gs_U, \Xi)$ is identified with the profinite completion of its abstract fundamental group $(\Gs_U, \Xi)$ via \cite[Proposition 6.5.6]{Ribes_Graphs}, and the epimorphism $\pi_U$ factors as
    \[
    \PA[(\Gs_U, \Xi)] \xrightarrow{\iota_U} \PP[(\Gs_U, \Xi)] \xrightarrow{\widehat{\pi_U}} \G/U
    \]
    where $\iota_U \colon \PA[(\Gs_U, \Xi)] \hookrightarrow \PP[(\Gs_U, \Xi)]$ is the profinite completion map---it is injective as the virtually free group $\PA[(\Gs_U, \Xi)]$ is residually finite. The continuous map $q_U \colon G \to \PP[(\Gs_U, \Xi)]$ induced by the quotients of vertex groups agrees on the dense subset $\psi(\G)$ with the composition $\widehat{\pi} \circ p_U \circ \iota$, so we find that the diagram
    \[
    \begin{tikzcd}
        \Gh \arrow[dd, "\widehat \psi"', bend right] \arrow[rrr, "p_U", two heads] &&& \G/U\\
        & \G \arrow[ld, "\psi"'] \arrow[lu, "\iota"', hook] \arrow[r, "\varpi_U"] & {\PA[(\Gs_U,\Xi)]} \arrow[rd, "\iota_U", hook] \arrow[ru, "\pi_U", two heads] &\\
        G \arrow[rrr, "q_U", two heads]&&& {\PP[(\Gs_U,\Xi)]} \arrow[uu, "\widehat{\pi_U}"', two heads, bend right]
    \end{tikzcd}
    \]
    commutes. Now $\widehat{\pi_U}q_U\widehat\psi(g) = p_U(g) \neq 1$, whence also $\widehat\psi(g) \neq 1$. Thus $\widehat \psi$ is injective and an isomorphism of profinite groups.
    
    Finally, consider the postulated diagram of cohomology groups. Its columns are the long exact sequences given in Proposition~\ref{Prop::Mayer-Vietoris}. Observe that:
    \begin{enumerate}[(1)]
    	\item the horizontal maps $\mu_V$ and $\mu_E$ are induced by restrictions of the profinite completion map $\iota \colon \G \to \Gh$; and
    	\item the vertical maps maps $i^*$ and $\hbar$ are induced by the canonical morphisms $i_v \colon \Gs(v) \to \G$ (respectively, $\Gsc(v) \to \Gh$), as well as linear combinations of the monomorphisms $\partial_{e,i} \colon \Gs(e) \to \G$ (resp. $\partial_{e,i} \colon \Gsc(e) \to \Gh$) and twists by stable letters.
    \end{enumerate} 
    It follows from the construction of the profinite graph of groups $(\Gs, \Xi)$ that the maps in (1) commute with the maps in (2), so the upper two squares of the diagram must commute. The boundary homomorphisms $\delta$ derive naturally from the short exact sequence in \cite[Theorem 1]{Chiswell}, whence the bottom square of the diagram must commute as well (see e.g. \cite[I.0.4]{brown_cohomology}). This completes the proof.
\end{proof}
Using the above theorem, we will be able to work largely with procyclic groups when proving (or disproving) cohomological separability for GBS groups. If such a group has no torsion it will be isomorphic to the completion of $\Z$ with respect to some formation of finite groups by classification of procyclic groups \cite[Proposition 2.3.8 and Proposition 2.7.1]{Bible}. To study the cohomology of such completions, we shall require the following result, which generalises Lemma~\ref{Lem::DimOne} to completions with respect to narrower formations of finite groups.
\begin{lemma}
    Let $\Cc$ be a formation of finite groups, $\G$ a group, and $\iota \colon \G \to \Gc$ its pro-$\Cc$ completion. Let $M \in \Cc$ be a $\Gc$-module, considered also as a $\G$-module with action via the map $\iota$. For any derivation $f \colon \Gamma \to M$, there exists a unique continuous derivation $\fh \colon \Gc \to M$ satisfying $\fh \iota = f$.
\end{lemma}
\begin{proof}   
    Recall that derivations $\G \to M$ are equivalently homomorphic sections $\G \to M \rtimes \G$ of the proj  ection $\operatorname{pr_2} \colon M \rtimes \G \to \G$. It follows that $f$ induces a homomorphism
    \[
    f' \colon \G \xrightarrow{f} M \rtimes \G \xrightarrow{(\operatorname{id},\iota)} M \rtimes \Gc 
    \]
    where the action of $\Gc \curvearrowright M$ arises canonically as the composition of the map $\iota$ and the action $\G \curvearrowright M$. Using the universal property of $\Gc$ (see \cite[Lemma 3.2.1]{Bible}, for example), we obtain a unique homomorphism
    \[
    \fh \colon \Gc \to M \rtimes \Gc
    \] satisfying $\fh \iota = f'$. The composition $\operatorname{pr}_2 \fh$ agrees with $\operatorname{id}_{\Gc}$ on the dense subset $\iota(\G) \subseteq \Gc$, so it must agree on all of $\Gc$. Thus $\fh$ is equivalently a continuous derivation $\fh \colon \Gc \to M$ satisfying $\fh\iota = f$. Its uniqueness follows again from the density of $\iota(\G)$ in $\Gc$.
\end{proof}
Given a set of primes $\Sc$, we shall write $\Cc(\Sc)$ for the formation of groups whose orders are divisible only by primes in $\Sc$. Given a group $\G$, we shall write $\Gsp$ for the completion of $\G$ with respect to the pro-$\Cc(\Sc)$ topology, which is commonly referred to as the ``pro-$\Sc$ completion''. For instance, the Chinese Remainder Theorem yields $\Z_{\Sc} = \prod_{p \in \Sc} \Z_p$. In this language, the previous proposition reduces to:
\begin{corollary}\label{Cor::ContDer}
    Let $\Sc$ be a set of primes, $\G$ a group and $\iota \colon \G \to \Gsp$ its pro-$\Sc$ completion. Let $M$ be a finite $\Gsp$-module whose order is a product of primes in $\Sc$, with $\G$-module structure via the map $\iota$. For any derivation $f \colon \G \to M$, there exists a unique continuous derivation $\fh \colon \Gsp \to M$ satisfying $\fh \iota = f$.
\end{corollary}
As a conclusion of this sequence of results we have the following, which we will require for the torsion free case.

\begin{lemma}\label{Lem::Surj}
	Let $\G = \PA$ be a generalised Baumslag--Solitar group with profinite completion $\iota \colon \G \to \Gh$. Assume that there is a collection of primes $\Sc$ such that the profinite topology induced by $\G$ on each edge group is the full pro-$\Sc$ topology. The map \[
	\mu_E = \prod_{e \in E(\Xi)} (\iota\at{\Gs(e)})^* \colon \bigoplus_{e \in E(\Xi)} H^1(\Gsc(e),M) \longrightarrow \bigoplus_{e \in E(\Xi)} H^1(\Gs(e),M)
	\] is surjective for any finite $\G$-module $M$ whose order is a product of primes in $\Sc$.
\end{lemma}
\begin{proof}
	The map $\mu_E$ decomposes into a product of maps \[\mu_E^e = (\iota\at{\Gs(e)})^* \colon H^1(\Gsc(e),M) \to H^1(\Gs(e),M)\] where $e \in E(\Xi)$ ranges through all edges of $\Xi$. Hence, it will suffice to show that each individual $\mu_{E}^e$ is surjective. Indeed, let $M$ be a finite $\G$-module whose order is divisible only by primes in $\Sc$, and let $[f] \in H^1(\Gs(e),M)$ be a class represented by some derivation $f \colon \Gs(e) \to M$. The finite $\G$-module structure on $M$ induces a finite $\Gh$-module structure via the universal property of the profinite completion, which in turn induces a finite $\Gsc(e)$-module structure on $M$ via restriction. By assumption, there is an isomorphism $\Gs(e) \cong \Z_{\widehat \Sc}$, and the restriction of the profinite completion map $\iota \at{\Gs(e)}$ of $\Gamma$ translates to the pro-$\Sc$ completion map $\iota_{\Sc} \colon \Z \to \Z_{\widehat \Sc}$ under this isomorphism. Now Corollary~\ref{Cor::ContDer} yields a continuous derivation $\fh \colon \Gsc(e) \to M$ such that $\fh \iota_{\Sc} = f$, or equivalently, $\fh \iota\at{\Gs(e)} = f$. Thus $\mu_E^e([\fh]) = f$, as required.
\end{proof}

Finally, we have the following technical lemma which will be useful in detecting torsion in the subgroups $\overline{\mathcal{G}}(u)\leq\Gh$ for $v\in v(\Xi)$.

\begin{lemma} \label{Lem::DifferenceIsTorsion}
    Let $\G = \PA$ be a generalised Baumslag--Solitar group and write $\iota\colon\G\rightarrow\Gh$ to denote its profinite completion. If $\Gh$ is torsion-free then the profinite topology induced by $\G$ on each of its edge and vertex groups agrees.
\end{lemma}
\begin{proof}
    Let $v\in V(\Xi)$ be any vertex, and $e\in E(\Xi)$ an edge incident on $v$, and let $a$ be a generator of $\Gs(v)$. We may assume without loss of generality that $v=d_0(e)$.

    Assume that there exists some finite quotient $\phi\colon\G\rightarrow Q$ of $\G$ in which $|\phi(\Gs(v))|=p^k$ for some prime $p$ and $k\in\Z_{>1}$. Then by the classification of procyclic groups \cite[Proposition 2.3.8 and Proposition 2.7.1]{Bible} either $\Gsc(v)$ has torsion, a contradiction, or $\Gsc(v)$ contains a full copy of the $p$-adic integers $\Z_p$, so we must have $\Z_p\leq \Gsc(v)$. It follows that there exists a finite quotient $\theta\colon\G\rightarrow F$ of $\G$ in which $|\theta(\Gs(v))|=p^{k+\nu_p(\yi_0(e))}$. Then $\theta(\Gs(v))$ is cyclic, and there exists some generator $\omega\in \theta(\Gs(v))$ such that $\theta(a)=\omega$. The subgroup $\theta(\Gs(e))\leq \theta(\Gs(v))$ is then generated by $\omega^{\yi_0(e)}$, which has order $|\omega|-\nu_p(\yi_0(e))=k$. Thus $|\theta(\Gs(e))|=p^k$, so by the classification of procyclic groups and the fact that $\Gh$ is torsion free we must have that $\Z_p\leq \Gsc(e)$.
    
    Conversely, if $\Gs(v)$ has no induced $p$-quotients then neither does $\Gs(e)$ as $\Gsc(e)\leq\Gsc(v)$, so it follows that $\Gsc(e)\cong\Gsc(v)$, again by classification of procyclic groups.
    Thus, for all $v\in V(\Xi)$ and $e\in E(\Xi)$ incident on $v$ we have that $\Gsc(e)\cong\Gsc(v)$, and so the result then follows by the connectivity of $\Xi$.
\end{proof}

\section{The Cycle Case}
\label{Sec::BS}
In this section, we shall decide the cohomological separability of generalised Baumslag--Solitar groups over a cycle $\Xi$, in terms of the prime decomposition of the augmentation products $n = n(\Xi)$ and $m = m(\Xi)$ of the cycle (Theorem~\ref{Thm::MTALoopCase}). Given two integers $n,m \in \mathbb{Z}$, we shall define the \emph{isocracy locus} as the collection of primes
\[
\Ic(n,m) = \{p \,\, \mathrm{prime} \mid \nu_p(n) = \nu_p(m)\}
\]
that is, those primes $p$ which divide $n$ and $m$ with equal (possibly null) power. We obtain the following three results, which characterise the profinite topology induced by $\G$ on its edge and vertex groups in terms of the primes contained in $\Ic(n,m)$.
\begin{lemma}\label{Lem::Mini_Closure}
    Let $\G = \PA$ be a generalised Baumslag--Solitar group over a graph $\Xi$ with a single cycle $\Upsilon \subseteq \Xi$, let $\iota \colon \G \to \Gh$ be its profinite completion, and assume that the augmentation products of the cycle $n = n(\Upsilon)$ and $m = m(\Upsilon)$ are isocratic. For any edge $e \in E(\Xi)$, the profinite topology induced by $\G$ on the edge group $\Gs(e)$ contains the full pro-$\Ic(n,m)$ topology, i.e. there is a monomorphism
    \[
        \Z_{\widehat{\Ic(n,m)}} \hookrightarrow \overline{\iota(\Gs(e))}
    \]
    which commutes with the canonical maps $\Z \to \Z_{\widehat{\Ic(n,m)}}$ and $\Gs(e)\to \overline{\iota(\Gs(e))}$.
\end{lemma}
\begin{proof}
    Write $\Chi = \Gs(e)$ and $K = \overline{\iota(\Gs(e))}$. By the classification of procyclic groups \cite[Proposition 2.3.8]{Bible}, there is an isomorphism $K \cong \prod_{p \, \mathrm{prime}} K_p$ for some collection of procyclic pro-$p$ groups $K_p$. This isomorphism is canonical in the sense that there is a commutative diagram
    \[
    \begin{tikzcd}
        \Z_{\widehat\Sc} \arrow[r, hook] & K \\
        \Z \arrow[u, hook] \arrow[r, "\vartheta_e"] & \Chi \arrow[u]
    \end{tikzcd}
    \]
    whenever $\Sc$ is a collection of primes for which $K_p \cong \Z_p$. We shall demonstrate that this is the case for $\Sc = \Ic(n,m)$, meaning that for any prime $p \in \Ic(n,m)$ and positive integer $l$, there exists a map $f \colon \G \to Q$ such that $C = f(\Chi)$ is a cyclic $p$-group of order $p^l$. Fix such a prime $p \in \Ic(n,m)$ and let $k = \nu_p(n) = \nu_p(m)$.
    
    Choose any edge $e \in E(\Xi)$ and let $\Theta$ be a spanning tree of $\Xi$, oriented in such a way that $\Upsilon$ obtains a clockwise orientation and edges in $\Xi - \Upsilon$ are oriented away from $\Upsilon$. Let $\varepsilon = \varepsilon_{p, w}$ be the power counting function associated to the prime $p$ based at some vertex $w \in V(\Xi)$ with respect to the tree $\Theta$, as per (\ref{Eq::Epsilon}). As $n$ and $m$ are isocratic, the choice of vertex $w \in V(\Upsilon)$ only alters $\varepsilon_{p,w}$ by a constant, so we may choose $w \in V(\Xi)$ which minimises $\varepsilon_{p, w}$ and assume henceforth that $\varepsilon = \varepsilon_{p,w}$ is non-negative on all of $\Xi$. Denote $v = d_0(e)$ and $b = \nu_p(\yi_0(e))$. Let $C$ be the cyclic group of order $p^{l + \varepsilon(v) + b}$ and $Q = C \rtimes \operatorname{Aut}(C)$. We shall define a homomorphism $f \colon \G \to Q$ on the generators $A = \{a_z : z \in V(\Xi)\}$ and $\{t\}$ of $\G$, arguing on $A$ inductively on generators corresponding to the vertices of $\Xi$, ordered by their edge distance (in the tree $\Theta$) from $w$.
		
	For the base case, define the helper variable $c_w = 1$ and set $f(a_w) = (c_w, \operatorname{id}_C)$. For the inductive step, assume that we have defined $f(a_z)$ and $c_z$ on all generators $a_z$ indexed by vertices $z \in V(\Xi)$ with edge distance $d(z,w) = \tau > 0$. Given a vertex $y \in V(\Xi)$ with distance $d(y,w) = \tau + 1$, denote by $u \in E(\Theta)$ the unique edge incident at $y$ in the geodesic $[w,y] \subseteq \Theta$ and $z$ for its other endpoint. Write $\yi_y$ and $\yi_z$ for the inclusion indices of $u$ at the vertices $y$ and $z$, respectively. Moreover, write $\yi_y = p^{\nu_p(\yi_y)} \cdot  \alpha_y$ and $\yi_z = p^{\nu_p(\yi_z)} \cdot \alpha_z$ for integers $\alpha_y,\alpha_z$ coprime to $p$. Define
	\[
	   c_y = \alpha_y^{-1} \alpha_z c_z \in C
	\]
	and set \[
        f(a_y) = \left(p^{\varepsilon(y)} c_y \, , \, \operatorname{id}_C\right)
    \]
    which is well-defined as the induction started with the $\varepsilon$-minimal vertex $w$. Finally, consider the extremal vertices the extremal vertices $y^+, y^-$ in $\Theta \cap \Upsilon$ and write $\yi_-, \yi_+$ for the inclusion indices of the unique edge in $\Xi - \Theta$ at its respective endpoints. As $p\in\Ic(n, m)$, we must have $\varepsilon(y^-) + \nu_p(\yi_-) = \varepsilon(y^+) + \nu_p(\yi_{+})$ by (\ref{eq:epsilonFormula}), whence derives an equality of orders \[
    |p^{\varepsilon(y^+)}c_{y^+}|^{\yi_+} = |p^{\varepsilon(y^-)}c_{y^-}|^{\yi_-}
    \] in $C$. Hence, there is an isomorphism $\phi \in \operatorname{Aut}(C)$ satisfying
    \[
    \phi \colon (p^{\varepsilon(y^+)}c_{y^+})^{\yi_{+}} \mapsto (p^{\varepsilon(y^-)}c_{y^-})^{\yi_{-}}
    \]
    and we define $f(t) = (0,\phi)$. One verifies that the above definition preserve the relations of the group presentation in (\ref{Eq::CanonicalPresentation}) and hence extend to a homomorphism $f \colon \G \to Q$. Moreover, one finds that $f(\Chi)$ is cyclic of order $|C| / p^{\varepsilon(v) + b} = p^l$, as per (\ref{Eq::Order}). It follows that $\Chi$ has $p$-torsion quotients of arbitrary size and $K_p \cong \Z_p$.
\end{proof}
Regarding fiber groups of points on the cycle $\Upsilon \subseteq \Xi$, we may sharpen Lemma~\ref{Lem::Mini_Closure} to an isomorphism. We obtain the following description of the closure of fibers in the profinite completion $\Gh$, which shall prove an instrumental tool for determining the cohomology of $\Gh$.
\begin{proposition}\label{Prop::Closure}
    Let $\G = \PA$ be a generalised Baumslag--Solitar group over a cycle $\Xi$, let $\iota \colon \G \to \Gh$ be its profinite completion, and assume that the augmentation products $n = n(\Xi)$ and $m = m(\Xi)$ are isocratic. For any $x \in \Xi$, the profinite topology induced by $\G$ on the fiber $\Gs(x)$ is the full pro-$\Ic(n,m)$ topology, i.e. there is an isomorphism of profinite groups
	\[
	\overline{\iota(\Gs(x))}  \cong \Z_{\widehat{\Ic(n,m)}}
	\]
	which commutes with the canonical maps $\Gs(x) \to \overline{\iota(\Gs(x))}$ and $\Z \to \Z_{\Ic(n,m)}$.
	\end{proposition}
	\begin{proof}
		Choose a vertex $v \in V(\Xi)$ and an edge $e \in E(\Xi)$ such that $x \in \{e,v\}$ and $e$ is incident at $v$. Write $\Chi := \Gs(e) \leq \Gs(v) =: \Delta$, as well as $K := \overline{\iota(\Gs(e))} \leq \overline{\iota(\Gs(v))} =: H$. Using Lemma~\ref{Lem::Mini_Closure} and the classification of procyclic groups \cite[Proposition 2.3.8]{Bible}, we obtain a composition of maps
        \[
        \Z_{\widehat{\Ic(n,m)}} \cong \prod_{p \in \Ic(n,m)} \Z_p \hookrightarrow K \cong \prod_{p \, \mathrm{prime}} K_p \hookrightarrow H \cong \prod_{p \, \mathrm{prime}} H_p
        \]
        for some collection of procyclic pro-$p$ groups $K_p$ and $H_p$. This composition is canonical in the sense that the diagram
        \[
            \begin{tikzcd}
                \Z_{\widehat{\Ic(n,m)}} \arrow[r, hook] & K \\
                \Z \arrow[u, hook] \arrow[r, "\vartheta_v"] & \Delta \arrow[u]
            \end{tikzcd}
        \]
	commutes. Hence, it will suffice to show that $H_p = 1$ whenever $p \notin \Ic(n,m)$.
		Fix a prime $p \notin \Ic(n,m)$, so that $p$ divides exactly one of $n$ and $m$; w.l.o.g. $p$ divides $n$ and not $m$. It follows from the presentation of $\G$ in (\ref{Eq::CanonicalPresentation}) that $a_v^{n} = ta_v^{\pm m}t^{-1}$ holds in $\G$, where the sign depends on the sign of the edge inclusions. Then any quotient $f \colon \G \to Q$ with $f(\Delta) = C_p = C$ cyclic of order $p$ must have $f(a_v^{n}) = 0$, as $n \equiv 0$ modulo $p$. However, $m$ is coprime to $p$, so any such quotient must also satisfy $f(\langle a^{m}\rangle) = C$. However, this implies $f(t) C f(t^{-1}) =  f(\langle t a^{m} t^{-1}\rangle)  =  f(\langle a^{n}\rangle) = 1$, a contradiction. We conclude that $\Delta$ cannot have a non-trivial $p$-torsion quotient deriving from $\G$, and $H_p = 1$ holds as postulated.
	\end{proof}
 In the case of a Baumslag--Solitar group on coprime integers $n,m$, Proposition~\ref{Prop::Closure} together with Proposition~\ref{Prop::ComDiag} yields the following well-known example; we invite the reader to regard Proposition~\ref{Prop::Closure} as a broad generalisation thereof.
 \begin{example}\label{Ex::coprimeBS}
    Let $n,m \in \mathbb{Z}_{\geq 1}$ be coprime integers greater than 1. The profinite completion of the Baumslag--Solitar group $\G = \operatorname{BS}(n,m)$ is
    \[
    \Gh \cong \Lh \rtimes_{\widehat\psi} \pZ \cong \prod_{p \in \Ic(n,m)} \Z_p \rtimes_{\widehat\psi} \pZ
    \]
    where $\Lambda = \Z[\frac{1}{nm}]$ and the homomorphism $\widehat\psi \colon \pZ \to \operatorname{Aut}(\Lh)$ is generated by the multiplication automorphism $\widehat\psi(1)\colon x \mapsto \frac{m}{n} \cdot x$.
\end{example}
In the case of a GBS group $\G$ with coprime augmentation products $n$ and $m$, the conjunction of Lemma~\ref{Lem::Surj} and Proposition~\ref{Prop::Closure} will allow us to lift derivations $\G \to M$ to the profinite completion $\Gh$ whenever $M$ has order coprime to $nm$. The following lemma shows that it will suffice to consider such modules only.
\begin{lemma}\label{Lem::Vanishing}
    Let $\G = \PA$ be a generalised Baumslag--Solitar group over a cycle $\Xi$, and assume that the augmentation products $n = n(\Xi)$, $m = m(\Xi)$ are coprime. If $M$ is a finite $\G$-module whose order is the power of a prime $p$ which divides exactly one of $n$ or $m$, then $H^2(\G,M) = 0$.
\end{lemma}
\begin{proof} 
    We may assume without loss of generality that $p$ divides $m$ and not $n$, and that $p$ in fact divides $\yi_0(e_s)$. We may also assume that every edge $e\in E(\Xi)$ has distinct endpoints. Label the edges and vertices of $\Xi$ as in Figure~\ref{Fig:HEXAGON}, and take the generating set for $\G$ given in the standard presentation (\ref{Eq::CanonicalPresentation}).
    
    We proceed by induction on $k$, where $p^k$ is the largest order of an element of $M$ when considered as an abelian group. For the base case let $k=1$, so $M\cong\Fp^l$ for some natural number $l$. By Proposition~\ref{Prop::Mayer-Vietoris} it suffices to show that the map
    \[
    \hbar \colon \bigoplus_{v \in V(\Xi)} H^1(\Gs(v),M) \to \bigoplus_{e \in E(\Xi)} H^1(\Gs(e),M)
    \]
    induced by the map on derivations given by
    \[
    \hbar\left((f)_{v \in V(\Xi)} \right) \colon \left((x_e)_{e \in E(\Xi)} \right) \mapsto \left(f_{d_1(e)}(\partial_1(x_e)) - t_e \cdot f_{d_0(e)}(\partial_0(x_e)) \right)
    \]
    is surjective. Let $\phi\colon\G\rightarrow \operatorname{Aut}(M)$ be the $\Gamma$-module structure on $M$. For each $v_i\in V(\Xi)$ the $\langle a_i\rangle$-module structure on $M$ is determined by the image $\phi(\langle a_i\rangle)$, which will be a cyclic group of order not divisible by any factor of $nm$ using Proposition~\ref{Prop::Closure}. It follows that taking $n_{i-1}^{\mathrm{th}}$ or $m_{i}^{\mathrm{th}}$ powers of elements of $\phi(\langle a_i\rangle)$ represents an automorphism for all $1\leq i\leq s$, so 
    \[\phi(\langle a_i\rangle)=\phi(\langle a_i^{n_{i-1}}\rangle)=\phi(\langle a_i^{m_{i}}\rangle)\]
    and the induced $\mathcal{G}(v_i)$-module structure on $M$ is (non-canonically) isomorphic to the induced $\mathcal{G}(e_i)$ and $\mathcal{G}(e_{i-1})$-module structures for all $1\leq i\leq s$. Thus, for all vertices $v_i\in V(\Xi)$, 
    \[H^1(\Gs(v_i),M)\cong H^1(\Gs(e_i),M)\cong H^1(\Gs(e_{i-1}),M)\cong M/(a_i-1)M\]
    using that each vertex and edge group is isomorphic to $\Z$, and we may view each vertex and edge group as isomorphic to $M/(a_1-1)M$ where convenient by connectivity of $\Xi$.
    
    Given subsets $U\subseteq E(\Xi)$ and $X\subseteq V(\Xi)$, we define the map $\hbar_{X\to U}$ as the composition $\operatorname{pr}_U\circ \hbar$ restricted to the domain $\bigoplus_{v \in X} H^1(\Gs(v),M)$, where \[\operatorname{pr}_U\colon \bigoplus_{e \in E(\Xi)} H^1(\Gs(e),M)\rightarrow \bigoplus_{e \in U} H^1(\Gs(e),M)\] is the natural projection. Let $v_i\in V(\Xi)$, $e_j\in E(\Xi)$, let $b_j\in \Gs(e_j)\leq\Gs(v_j)$ be the generator of $\Gs(e_j)$ identified with $a_j^{n_j}$ in $\G$ and let $f\colon \Gs(v_i)\rightarrow M$ be a derivation. Write $x = f(a_i) \in M$, which determines $f$ uniquely. On the level of derivations, $\hbar_{v_i\to e_j}$ is generated by
    \[\hbar_{v_i\to e_j}(f)(b_j)=\begin{cases}
        -t_{e_j}f(a_i^{m_j}), & v_i=d_0(e_j) \\
        f(a_i^{n_j}), & v_i=d_1(e_j) \\
        0 & \text{else}
    \end{cases}\]
    recalling that no edge in $\Xi$ is a loop, and that any derivation $\Z\rightarrow M$ is entirely defined by the image of a generator. Assume first that $m_j, n_j>0$. Then, by the derivation law,
    \begin{align*}
    f(a_i^{m_j})&=f(a_i)+a_i\cdot f\left(a_i^{(m_j-1)}\right)\\
    &= \ldots \\
    &=\left(1+a_i+\ldots+a_i^{(m_j-1)}\right)\cdot x
    \end{align*}
    and a similar formula holds for $n_j$. Noting additionally that $a_i$ acts trivially on $M/(a_i-1)M$ and $t_e=_\G 1$ for all $j\neq s$, we thus obtain
    \[\hbar_{v_i\to e_j}(x)=\begin{cases}
        -tm_j\cdot x, & v_i=d_0(e_j)\text{ and } j=s \\
        -m_j\cdot x, & v_i=d_0(e_j)\text{ and } j\neq s \\
        n_j\cdot x, & v_i=d_1(e_j) \\
        0 & \text{else}
    \end{cases}\]
    for any $x \in M/(a_j - 1) \cong H^1(\Gs(e_j),M)$, where $t:=t_s$ as in the presentation (\ref{Eq::CanonicalPresentation}). The case where $m_j<0$ or $n_j<0$ is similar, as, for instance, $m_j<0$ gives 
    \[f(a_i^{m_j})=- (a_i^{m_j}+a_i^{m_j+1}+...+a_i^{-1})x\]
    which contains the same number of terms. We thus obtain the following formula for $\hbar_{v_i\to e_j}$ in the general case:
    \begin{equation}\hbar_{v_i\to e_i}(x)=\begin{cases}
        \mp t(\yi_0(e_j))\cdot x, & v_i=d_0(e_j)\text{ and } j=s \\
        \mp \yi_0(e_j)\cdot x, & v_i=d_0(e_j)\text{ and } j\neq s \\
        \pm \yi_1(e_j)\cdot x, & v_i=d_1(e_j) \\
        0 & \text{else}
    \end{cases}\end{equation}
    where the sing depends on the sign of $m_j$ and $n_j$. It follows from this and the observation that all edge and vertex cohomologies are isomorphic to $M/(a_1-1)M$ for all $i$ that $\hbar$ is given by the matrix 
    \[A_{\hbar}=\begin{pmatrix}
        \mp\yi_0(e_1) & \pm\yi_1(e_1) & 0 &\ldots& 0 & 0\\
        0 & \mp\yi_0(e_2) & \pm\yi_1(e_2)  &\ldots& 0 & 0\\
        0 & 0 & \mp\yi_0(e_3) &\ldots& \pm\yi_1(e_{s-2}) & 0\\
        \vdots & \vdots & \vdots &\ddots& \mp\yi_0(e_{s-1}) & \pm\yi_1(e_{s-1}) \\
        \pm\yi_1(e_s) & 0 & 0 & 0 & 0 & \mp t(\yi_0(e_s))        
    \end{pmatrix}\]
    over $\Z\G$. By assumption, $p$ divides $\yi_0(e_s)$, so multiplication by $\yi_0(e_s)$ represents the zero map in $M$ or any quotient thereof. It follows that the $\Fp$-linear transformation $\bigoplus_{v \in V(\Xi)} H^1(\Gs(v),M) \xrightarrow{A_\hbar} \bigoplus_{e \in E(\Xi)} H^1(\Gs(e),M)$ has determinant given by the class $n = \pm \prod_{1\leq i\leq s}\yi_1(e_i) \mod p$, which is non-zero as $\gcd(n,p) = 1$. We infer that $\hbar$ is surjective, whence $H^2(\G, M) = 0$, as required.
    
    For the inductive step, let $k>1$ be an integer and assume that $H^2(\G,M) = 0$ whenever $M$ is a finite $\G$-module such that the order of any element in $M$ divides $p^k$. Let $M$ be a $\G$-module with order $|M| = p^{k+1}$, and let $\Phi(M)$ be the Frattini subgroup of $M$. As $\Phi(M)$ is a characteristic subgroup, it is also a $\G$-submodule. Hence we obtain the short exact sequence
    \[0\rightarrow \Phi(M)\rightarrow M\rightarrow M/\Phi(M)\rightarrow 0\]
    of $\G$-modules, which induces a long exact sequence of cohomology groups, including the fragment
    \[\ldots\rightarrow H^2(\G, \Phi(M))\rightarrow H^2(\G, M)\rightarrow H^2(\G, M/\Phi(M))\rightarrow\ldots \]
    in dimension 2. By construction, elements in $\Phi(M)$ and $M/\Phi(M)$ have order at most $p^k$, so the inductive hypothesis yields $H^2(\G, \Phi(M)) = H^2(\G, M/\Phi(M))= 0$. By exactness, it follows that $H^2(\G, M)=0$ holds as well. We conclude that $H^2(\G,M) = 0$ for all $\G$-modules $M$ whose order is a power of $p$, as postulated.
\end{proof}
The following allows one to derive cohomological separability from the conjunction of Proposition~\ref{Prop::Closure} and Lemma~\ref{Lem::Vanishing}.
\begin{proposition}\label{Prop::PositiveCase}
   Let $\G = \PA$ be a generalised Baumslag--Solitar group. If there is a collection of primes $\Ic$ satisfying:
   \begin{enumerate}[(a)]
       \item for any edge $e \in E(\Xi)$, the profinite topology induced by $\G$ on the edge group $\Gs(e)$ is its full pro-$\Ic$ topology, and
       \item if $M$ is a finite $\G$-module whose order is coprime to $\Ic$ then $H^2(\G,M) = 0$,
   \end{enumerate}
   then $\G$ has separable cohomology.
\end{proposition}
\begin{proof}
    Combining Corollary \ref{Cor::CohomDim} with Lemmata \ref{Lem::DimOne} and \ref{Lem::Surjections}, we observe that it will be sufficient to show the profinite completion map $\iota \colon \G \to \Gh$ induces a surjection
    \[
        \iota^* \colon H^2(\Gh,M) \twoheadrightarrow H^2(\G, M)
    \]
    on cohomology in dimension 2. Let $M$ be a finite $\G$-module, which decomposes as a direct sum of its Sylow-$p$ subgroups
    $
    M = \bigoplus_{p \, \mathrm{prime}} M_p
    $
    which is also a direct sum of $\G$-submodules as each $M_p$ is uniquely determined by its order and hence invariant under the action of $\G$. Observe that
    \[
    H^2(\G,M) = \bigoplus_{p \in \Ic} H^2(\G, M_p)
    \]
    by assumption (b) and the fact that cohomology commutes with finite direct sums of modules. We may thus reduce to the case where the order $M$ is the power of a prime in $\Ic$. Now Proposition~\ref{Prop::ComDiag} yields a commutative diagram
    \[\begin{tikzcd}
        {\bigoplus_{e \in E(\Xi)} H^1(\Gsc(e),M)} \arrow[d, "\delta"] \arrow[rr, "\mu_E"] && {\bigoplus_{e \in E(\Xi)} H^1(\Gs(e),M)} \arrow[d, "\delta"] \\
        {H^{2}(\Gh,M)} \arrow[d]\arrow[rr, "\iota^*"]                                 && {H^{2}(\G,M)} \arrow[d]                    \\
        0 && 0
        \end{tikzcd}
        \]
    with exact columns. Given assumption (a), we may appeal to Lemma~\ref{Lem::Surj} to find that $\mu_{E}$ is an epimorphism. Hence $\iota^* \delta = \delta \mu_E$ is an epimorphism as well, and so is $\iota^*$. This completes the proof.
\end{proof}
This contrasts with the case of a GBS group with isocratic but not coprime augmentation products. There, one can construct an explicit module witnessing the failure of cohomological separability: see Corollary~\ref{Cor::IsocraticNotCoprime} below. We outsource most of the proof to the following lemma, which also allows for a characterisation of the cohomological dimension of $\Gh$ in the isocratic case: see Corollary~\ref{Cor::ProfCohomDim} below.
\begin{lemma}\label{Lem::FunnyModule}
    Let $\G = \PA$ be a generalised Baumslag--Solitar group over a cycle $\Xi$ whose augmentation products $n = n(\Xi)$ and $m = m(\Xi)$ are isocratic but not coprime. Let $p$ be a prime dividing $nm$ and $q$ be a prime (possibly $q = p$) dividing $\gcd(n,m)$. There is a $\G$-module structure on $M = \Fp^q$ such that $H^2(\G,M) \neq 0$ holds. If, additionally, $p$ divides $\gcd(n,m)$, then $H^2(\Gh,M) \neq 0$ holds as well.
\end{lemma}
\begin{proof}
    Consider the presentation of $\Gamma$ given in (\ref{Eq::CanonicalPresentation}) along with the standard spanning tree $\Theta$ of $\Xi$, as well as the function $\varepsilon_{q, v}$ associated to the prime $q$ defined in (\ref{Eq::Epsilon}), with respect to the tree $\Theta$ and some vertex $v\in V(\Xi)$. Let $W \subseteq V(\Xi)$ be the subset of vertices which minimise the function $\varepsilon_{q, v}$, which is independent of $v$, and $Y \subseteq E(\Xi)$ be the set of those edges $e$ whose initial endpoint $v = d_0(e)$ lies in $W$ and such that $q$ divides the inclusion index $\yi_0(e)$. As $q$ divides both $n$ and $m$, and $\varepsilon_{q, v}$ increases at those edges whose initial inclusion index is divisible by $q$, we must have $Y \neq \emptyset$. Choose a basis $B$ for $M$ as a $\Fp$-vector space, and let $\alpha \in \Aut(M)$ be the automorphism induced by a cyclic permutation of $B$. Define $\phi \colon \G \to \Aut(M)$ as the map generated by
   \[
   \phi(a_i) = 
   \begin{cases}
        \alpha, & v_i \in W \\
        \operatorname{id}_M, & v_i \notin W
   \end{cases}
   \]
   and $\phi(t) = \operatorname{id}_M$, which, as one may verify using (\ref{Eq::CanonicalPresentation}), extends to a homomorphism $\phi \colon \G \to \Aut(M)$. Then $M$ forms a finite $\G$-module with action via $\phi$. We shall demonstrate that $M$ satisfies the postulated condition on cohomology.

   Recall the portion 
   \begin{equation}\label{Eq::Mayer--Vietoris--Portion}
       \ldots \to \bigoplus_{v \in V(\Xi)} H^1(\Gs(v),M) \xrightarrow{\hbar} \bigoplus_{e \in E(\Xi)} H^1(\Gs(e),M)  \xrightarrow{\delta} H^2(G,M) \to 0
   \end{equation}
   of the Mayer--Vietoris sequence given in Proposition~\ref{Prop::Mayer-Vietoris}, where the rightmost term vanishes as $\cd(\Z) = 1$. Given a vertex $v \in V(\Xi)$ with corresponding generator $a$, the $\Fp$-linear transformation of $M$ defined by $(a - 1)$ is the zero map if and only if $v \notin W$. Thus  for all $v\in W$ the inequation
   \[
   \dim_{\Fp}(H^1(\Gs(v),M)) = \dim_{\Fp}\left(\frac{M}{(a-1)M}\right) < \dim_{\Fp}(M) = q
   \]
   holds. On the other hand, suppose that $e \in E(\Xi)$ is an edge; write $v = d_0(e)$ for its initial endpoint, as well as $a,b$ for the generators of the groups $\Gs(v)$ and $\Gs(e)$, respectively. If $e \in Y$ then $b$ is conjugate to a $q$-power of some $a$, and hence the action of $b-1$ on $M$ is the zero map. It follows that
   \[
   \dim_{\Fp} \left( \frac{M}{(b-1)M} \right) = q\] holds whenever $e \in Y$. Conversely, we claim that there is an equality $\operatorname{rk}_{\Fp}(b-1) = \operatorname{rk}_{\Fp}(a-1)$ of ranks of $\Fp$-linear transformations of $M$ whenever $e \notin Y$. Indeed, if $v \notin W$, then $a$ acts trivially on $M$, and so does $b \sim a^{\yi_0(e)}$. If $v \in W$ but $\yi_0(e)$ is coprime to $q$, then $b$ is conjugate to $a^{\yi_0(e)}$ and the subgroups of $\alpha(\G)$ generated by $a$ and $a^{\yi_0(e)}$ are equal. Therefore, $a$ and $b$ have equal rank as $\Fp$-linear transformations of $M$, and so do the linear transformations $a-1$ and $b-1$. Putting everything together, we obtain
   \begin{align*}
   	\dim_{\Fp}&\left(\bigoplus_{e \in E(\Xi)} H^1(\Gs(e),M)\right) - \dim_{\Fp}\left(\bigoplus_{v \in V(\Xi)} H^1(\Gs(v),M)\right)= \\
   	&= \sum_{i = 1}^s \left[\dim_{\Fp}\left(\frac{M}{(b_i - 1)M} \right) - \dim_{\Fp}\left(\frac{M}{(a_i - 1)M}\right)\right] \\
   	&\geq \sum_{v_i \in Y} [q - (q - 1)] \\
   	&> 0
   \end{align*}
	where the final inequality holds as $Y \neq \emptyset$. It follows that $\hbar$ cannot be a surjection and $H^2(\G,M) \neq 0$. For the final statement, assume that $p$ divides $\gcd(n,m)$. Then Proposition~\ref{Prop::Closure} states that maps from the closures of vertex and edge groups in $\Gh$ to $M$ factor through the pro-$\{p,q\}$ completions of $\Z$. Hence the above argument holds analogously\footnote{Alternatively, the inclined reader may verify that the the analogous statement to Proposition~\ref{Prop::ComDiag} holds in any pro-$\pi$ category and that cohomological separability of $\Z$ implies cohomological separability of $\G$ in the pro-$\pi$ category whenever $\pi \subseteq \Ic(n,m)$, as per Proposition~\ref{Prop::Closure}. Writing $\pi = \{p,q\}$, one then obtains $H^2(\Gh,M) \cong H^2(\G_{\widehat{\pi}},M) \cong H^2(\G,M)$, where the first isomorphism is induced on cohomology by the canonical isomorphism of the pro-$\pi$ completions of $\G$ and $\Gh$.} for the induced $\Gh$-module $M$, and we find that $H^2(\Gh,M) \cong H^2(\G,M) \neq 0$.
\end{proof}
\begin{corollary}
		\label{Cor::ProfCohomDim} If $\G$ is a generalised Baumslag--Solitar group over an isocratic cycle then $\cd(\G) = \cd(\Gh) = 2$.
	\end{corollary}
	\begin{proof}
		Let $\G = \PA$ be a generalised Baumslag--Solitar group over an isocratic cycle. We already know that $\cd(\G) = 2$, as $\cd(\G)\leq 2$ by Corollary~\ref{Cor::CohomDim} and $\cd(\G)\geq 2$ as $\G$ is not free. Analogously, observe
		\[
		\cd(\Gh) \leq \max\{\cd(\Gsc(v) : v \in V(\Xi)\} + 1 \leq 2
		\]
		where the first inequality follows from Proposition~\ref{Prop::Mayer-Vietoris} and the second holds by Proposition~\ref{Prop::Closure}. Conversely, we shall demonstrate that $\cd(\Gh) \geq 2$. If the augmentation products $n = n(\Xi)$ and $m = m(\Xi)$ are coprime, then \[
        \Gh \cong \widehat{\operatorname{BS}(n,m)} \cong \prod_{p \notin \Ic(n,m)} \Z_p \rtimes \pZ
        \]
        where the first isomorphism holds by \cite[Lemma 3]{Dudkin}, and the second isomorphism holds by Example \ref{Ex::coprimeBS}. Thus $\Gh$ contains the pro-$p$ subgroup $\Z_p \rtimes \Z_p$ for any $p \in \Ic(n,m)$, which is not projective, and so $\cd(\Gh) \geq 2$ by \cite[Corollary 7.7.6]{Bible}. On the other hand, suppose that the augmentation products are isocratic but not coprime, so there exists a prime number $p$ which divides both $n$ and $m$ with equal, non-zero power. Then Lemma~\ref{Lem::FunnyModule} with $q = p$ yields the result.
	\end{proof}
\begin{corollary}
    \label{Cor::IsocraticNotCoprime}
    Let $\G = \PA$ be a generalised Baumslag--Solitar group over a cycle $\Xi$ and write $n = n(\Xi)$ and $m = m(\Xi)$ for the augmentation products. Assume that there are primes $p,q$ such that $p$ divides one and only one of $n$ and $m$, while $q$ divides both with equal non-zero power. Then $\G$ does not have separable cohomology.
\end{corollary}
\begin{proof} 
    By Lemma~\ref{Lem::DifferenceIsTorsion}, we need only be concerned with the case where $\G$ induces the same profinite topology on all edge groups; by Proposition~\ref{Prop::Closure}, this is the full pro-$\Ic$ topology for $\Ic$ a set of primes which does not contain $p$. Let $M$ be the $\G$-module given in Lemma~\ref{Lem::FunnyModule} with associated primes $p$ and $q$, so that $H^2(\G,M) \neq 0$. For any edge $e \in E(\Xi)$, the profinite group $\Gsc(e)$ is a pro-$\Ic$ group, so in particular it cannot have any cohomology with coefficients in $M$, i.e.
	\[
	H^1(\Gsc(e), M) = 0
	\]
	as per Lemma~\ref{Lem::CoprimeCohomologyVanishes}. But then Proposition~\ref{Prop::Mayer-Vietoris} yields an exact sequence
	\[
	0 =  \bigoplus_{e \in E(\Xi)} H^1(\Gsc(e),M)  \to H^2(\Gh,M) \to 0
	\]
	whence $H^2(\Gh,M) = 0$. Hence there cannot exist an epimorphism $0 = H^2(\Gh,M) \to H^2(\G,M) \neq 0$ and $\G$ cannot have separable cohomology.
\end{proof}
Finally, we obtain the following result, which deals with the case where $n$ and $m$ are not isocratic and serves as a converse to Proposition~\ref{Prop::Closure}.
\begin{proposition} \label{prop::NonIsocraticCycle}
   Let $\G = \PA$ be a generalised Baumslag--Solitar group over a cycle $\Xi$, let $\iota \colon \G \to \Gh$ be its profinite completion. If $\Gh$ is torsion-free then the augmentation products $n = n(\Xi)$ and $m = m(\Xi)$ must be isocratic.
\end{proposition}
\begin{proof}
    Assume for a contradiction that $m$ and $n$ are not isocratic, so that there exists some prime $p$ satisfying $\nu_p(n), \nu_p(m)\neq 0$ and $\nu_p(n)\neq \nu_p(m)$. We may assume without loss of generality that $\nu_p(n)> \nu_p(m)$. We shall demonstrate that there exists a vertex $v \in V(\Xi)$ such that $H(v) = \overline{\iota(\Gs(v))}$ contains torsion. As all vertex groups are cyclic, the profinite group $H(v)$ is procyclic, so \cite[Proposition 2.3.8]{Bible} gives $H = \prod_{q \text{ prime}} H_q(v)$ for some collection of procyclic pro-$q$ groups $H_q(v)$. Hence, it will suffice to show that there is a vertex $v \in V(\Xi)$ with $H_p(v)$ finite cyclic, or equivalently, that the following properties holds:
    \begin{enumerate}[(a)]
        \item there is a finite quotient $\phi \colon \G \to Q$ such that $p$ divides $|\phi(\Gs(v))|$; and
        \item there is an integer $k >1$ such that $p^k$ does not divide the order $|\phi(\Gs(v))|$ in any finite quotient $\phi \colon \G \to Q$.
    \end{enumerate}
    Label and orient $\Xi$ as in Figure~\ref{Fig:HEXAGON} and consider the presentation given in (\ref{Eq::CanonicalPresentation}). We commence with (b), showing that $k = \nu_p(n)$ has this property. Let $v_j$ be any vertex of $\Xi$ and assume for contradiction that there exists a finite quotient $\phi\colon\G\rightarrow Q$ of $\G$ in which $\nu_p(|\phi(a_j)|) \geq k$. By (\ref{Eq::CanonicalPresentation}), the generator $a_j$ corresponding to $v_j$ has conjugate powers $a_j^n \sim a_j^{\pm m}$, where the sign depends on the signs of the various edge inclusions, so the images of these powers in $Q$ must have the same order. But $\nu_p(|\phi(a_j)^n|)= 0 < k - \nu_p(m) = \nu_p(|\phi(a_j)^m|)$, a contradiction. This shows (b).

    For (a), begin by partitioning the vertex set $V(\Xi)$ into maximal (possibly length zero) paths $\{P_1,...,P_t\}$ such that each $P_i$ has the following property: 
    \begin{itemize}
    \item[(i)] For all edges $e$ in $P_i$ the numbers $\yi_1(e)$ and $\yi_0(e)$ are coprime to $p$.
    \end{itemize}
    We may assume that the paths $\{P_1, ..., P_t\}$ are labelled cyclically and clockwise such that $v_1\in P_1$, and we may assume up to relabeling that the edge $e_s$ does not lie in any of our $P_i$'s. We claim that at least one of our paths $P_i$ must also have the following property:
    \begin{itemize}
    \item[(ii)] For the unique edge $e(P_i)$ leaving $P_i$, the inclusion index $\yi_0(e(P_i))$ is divisible by $p$, and for the unique edge $f(P_i)$ incoming to $P_i$, the inclusion index $\yi_1(f(P_i))$ is divisible by $p$.
    \end{itemize}

    Indeed, assume for contradiction that no path $P_i$ has property (ii). Since $p$ divides both $m$ and $n$ there exists some path $P_i$ such that $\yi_0(e(P_i))$ is divisible by $p$, so by assumption $\yi_1(f(P_i))$ is not divisible by $p$ else $P_i$ would have property (ii). Then $\yi_0(e(P_{i-1}))=\yi_0(f(P_i))$ is divisible $p$ by maximality of the $P_j$'s, and similarly $\yi_1(f(P_{i-1}))$ is not divisible by $p$. It follows by induction that for all $j$ the integer $\yi_0(e(P_j))$ is divisible by $p$ but $\yi_1(f(P_j))$ is not, but then $n$ is not divisible by $p$, a contradiction. Thus at least one of our $P_j$'s has property (ii), so pick one such $P_j$ and denote it by $P$. Define $e=e(P)$, $f=f(P)$.

    By property (i), we may proceed as in the proof of Lemma~\ref{Lem::Mini_Closure} to construct a homomorphism $\phi_P\colon\pi_1(\Gs|_P, P)\rightarrow C_p$ satisfying the property that $\phi_P|_{\Gs(v)}$ is a surjection for all $v\in V(P)$. Consider the map $\phi\colon\G\rightarrow C_p$ generated by
    \[\phi(a_i) =\begin{cases}
        \phi_P(a_i), & v_i\in P \\
        0, & v_i \notin P
    \end{cases}
    \]
    and $\phi(t) = 0$, where $t$ is the unique non-trivial stable letter in $\G$. One verifies that $\phi$ extends to a well-defined homomorphism using properties (i) and (ii). We conclude that (a) holds, yielding the required contradiction and hence the result.
\end{proof}
\begin{theorem}\label{Thm::MTALoopCase}
    Let $\G = \PA$ be a generalised Baumslag--Solitar group over a cycle $\Xi$ with augmentation products $n = n_0(\Xi)$ and $m = n_1(\Xi)$. Exactly one of the following cases holds:
    \begin{enumerate}
        \item the numbers $n,m$ are either coprime or satisfy $|m|=|n|$, in which case $\cd(\G) = \cd(\Gh) = 2$ and $\G$ has separable cohomology;
        \item the numbers $n,m$ are isocratic but not coprime and $|n| \neq |m|$, in which case $\cd(\G) = \cd(\Gh) = 2$ but $\G$ does not have separable cohomology; and
        \item the numbers $n,m$ are not isocratic, in which case $\cd(\G) = 2$ but $\cd(\Gh) = \infty$, and $\G$ does not have separable cohomology.
    \end{enumerate}
\end{theorem}
\begin{proof}
    For (1), assume first $|n| = |m|$, in which case $\G$ is virtually $F_n \times \Z$ by \cite[Proposition 2.6]{Levitt07}, so it is LERF (cf. \cite[Theorem 4]{Allenby73}) and in particular it is an efficient graph of groups whose vertex groups have separable cohomology. Thus $\G$ has separable cohomology by \cite[Proposition 3.1]{GJZ}. On the other hand, suppose that $|n| \neq |m|$ but that $n,m$ are coprime. Then $\cd(\G) = \cd(\Gh) = 2$ by Corollary~\ref{Cor::ProfCohomDim}. Moreover, the two conditions of Proposition~\ref{Prop::PositiveCase} are satisfied by Proposition~\ref{Prop::Closure} and Lemma~\ref{Lem::Vanishing}, respectively. We infer that $\G$ must also have separable cohomology in this case.

    For (2), suppose that the numbers $n,m$ are isocratic but not coprime, and that $|n| \neq |m|$. Then there exists a prime $p$ which divides one and only one of $n$ or $m$ and a prime $q$ with $\nu_q(n) = \nu_q(m) > 0$. The postulated result now follows from Corollary~\ref{Cor::ProfCohomDim} and Corollary~\ref{Cor::IsocraticNotCoprime}.

    Finally, (3) follows from the conjunction of Corollary~\ref{Cor::CohomDim} and Proposition~\ref{prop::NonIsocraticCycle}.
\end{proof}

Theorem~\ref{Thm::MTA} follows as an immediate corollary of the above theorem.

\section{The General Case}\label{Sec::GBS}
In this section, we consider the general case where $\Xi$ is any graph. It turns out that, surprisingly, generalised Baumslag--Solitar groups over graphs which are more complex than cycles almost never have separable cohomology, for reasons simpler than in the cycle case. The situation is illustrated by the following three lemmata, the first of which provides a restrictive sufficient condition under which cohomological separability will always hold.
\begin{lemma}\label{Lem::GoodIfFullTop}
    Let $\G = \PA$ be a generalised Baumslag--Solitar group over a graph $\Xi$ and let $\iota \colon \G \to \Gh$ be its profinite completion. If $\G$ induces the full profinite topology on all of its vertex groups, then $\G$ has separable cohomology.
\end{lemma}
\begin{proof}
    As before, we may combine Corollary \ref{Cor::CohomDim} with Lemmata \ref{Lem::DimOne} and \ref{Lem::Surjections}, to reduce to showing that the profinite completion map $\iota \colon \G \to \Gh$ induces a surjection
    \[
        \iota^* \colon H^2(\Gh,M) \twoheadrightarrow H^2(\G, M)
    \]
    on cohomology in dimension 2. Assume that $\G$ induces the full profinite topology on all of its vertex groups, and let $M$ be a $\G$-module. Then $\G$ also induces the full topology on each of its edge groups by the proof of Lemma~\ref{Lem::DifferenceIsTorsion}. Now Proposition~\ref{Prop::ComDiag} yields a commutative diagram, whereof a portion reads
    \[\begin{tikzcd}
        {\bigoplus_{e \in E(\Xi)} H^1(\Gsc(e),M)} \arrow[d, "\delta"] \arrow[rr, "\mu_E"] && {\bigoplus_{e \in E(\Xi)} H^1(\Gs(e),M)} \arrow[d, "\delta"] \\
        {H^{2}(\Gh,M)} \arrow[d]\arrow[rr, "\iota^*"]                                 && {H^{2}(\G,M)} \arrow[d]                    \\
        0 && 0
        \end{tikzcd}
        \]
    whose columns are exact. By assumption, the map $\mu_E$ agrees with the product of maps induced by the profinite completion $\Gs(e) \to \Gsc(e) = \widehat{\Gs(e)}$, so it is an epimorphism by Lemma~\ref{Lem::Surj}. Thus $\iota^*$ must also be an epimorphism, and $\G$ has separable cohomology.
\end{proof}
We are able to characterise when a GBS group induces the full profinite topology on its vertex groups in terms of its augmentation products.
\begin{lemma} \label{lem:aIffb}
    Let $\G = \PA$ be a generalised Baumslag--Solitar group over a graph $\Xi$. Then the following are equivalent:
    \begin{enumerate}[(a)]
        \item For every cycle $\Upsilon \subseteq \Xi$, the equation $|m(\Upsilon)| = |n(\Upsilon)|$ holds,
        \item The group $\G$ induces the full profinite topology on each of its vertex groups.
    \end{enumerate}
\end{lemma}
\begin{proof}
    First assume that their exists some cycle $\Upsilon \subseteq \Xi$ such that equation $m(\Upsilon) = n(\Upsilon)$ does not hold. Then there exists some prime $p$ with $\nu_p(m(\Upsilon))\neq \nu_p(n(\Upsilon))$. The profinite topology induced by $\G$ on $\Gs(v)$ for any $v\in V(\Upsilon)$ must be at least as coarse as that induced by $\PA[(\Gs|_{\Upsilon}, \Upsilon)]$, and by Proposition~\ref{Prop::Closure}  there exists at least one vertex $v\in V(\Upsilon)$ such that $\PA[(\Gs|_{\Upsilon}, \Upsilon)]$ does not induce on $\Gs(v)$ its full profinite topology. It follows that (b) implies (a).

    For the other direction, assume that for every cycle $\Upsilon \subseteq \Xi$, the equation $|m(\Upsilon)| = |n(\Upsilon)|$ holds. We proceed by induction on $\operatorname{rk}_\Z H_1(\Xi)$ to prove a slightly stronger result: for any generalised Baumslag--Solitar group $\G = \PA$ over a graph of groups $\ggx$ with property (a), for all vertices $v\in V(\Xi)$, for any prime $p$ and any $k\in\Z_{\geq 1}$, there exists some $l\in\Z_{\geq 1}$ and a map $\phi\colon \G\rightarrow C_{p^l}\rtimes \operatorname{Aut}(C_{p^l})$ such that:
    \begin{enumerate}[(i)]
        \item The image of $\Gs(v)$ under $\phi$ is a copy of $C_{p^k}$, i.e. $\phi(\Gs(v))\cong C_{p^k}$, and;
        \item For all $w\in V(\Xi)$ we have that $\phi(\Gs(w))\leq C_{p^l}$, the normal factor in the semidirect product $C_{p^l}\rtimes \operatorname{Aut}(C_{p^l})$.
    \end{enumerate}
    For the base case, assume that $H_1(\Xi) = 0$, so $\Xi$ is a tree and property (a) holds vacuously. In this case, let $v\in V(\Xi)$ be a vertex, let $p$ be a prime and $k\in\Z_{\geq 1}$. Since $\Xi$ is a tree we may define the power counting function $\varepsilon_{p, w}$ as in (\ref{Eq::Epsilon}) with respect to $\Xi$. We may assume that $w$ minimises $\varepsilon_{p, w}$. Choose $l = k+\varepsilon_{p, w}(v)$. Then, as in the proof of Lemma~\ref{Lem::Mini_Closure}, we may construct a homomorphism $\phi\colon\G\rightarrow C_{p^l}\rtimes \operatorname{Aut}(C_{p^l})$ with the desired properties.
    
    For the inductive step, let $\G$ be a generalised Baumslag--Solitar group over a graph $\Xi$ with $\operatorname{rk}_\Z(H_1(\Xi))=j>0$. Fix some spanning tree $\Theta$ of $\Xi$, let $e \in E(\Xi)-E(\Theta)$ be an edge and let $\Upsilon\subseteq \Xi$ be some cycle subgraph that contains $e$, which we may assume is labeled and oriented as in Figure~\ref{Fig:HEXAGON}. Then, consider the subgraph of groups $(\Gs|_{\Xi-e}, \Xi-e)$. The fundamental group $\Delta=\PA[(\Gs|_{\Xi-e}, \Xi-e)]\leq \G$ is a generalised Baumslag--Solitar group over $\Xi-e$, which contains every vertex of $\Xi$ and has $\operatorname{rk}_\Z(\Xi-e)<\operatorname{rk}_\Z(\Xi)$. Let $v\in V(\Xi)$, let $p$ be a prime and $k\in\Z_{\geq 1}$. Then by induction there exists some natural number $l$ and a map $\phi_\Delta\colon \Delta\rightarrow C_{p^l}\rtimes \operatorname{Aut}(C_{p^l})$ that satisfies properties (i) and (ii). 
    
    We will show that $\phi_\Delta$ extends to a map $\phi\colon \G\rightarrow C_{p^l}\rtimes \operatorname{Aut}(C_{p^l})$ which retains properties (i) and (ii). Indeed, consider $\G$ via the presentation defined with respect to the spanning tree $\Theta$, so that $\G=\PA[(\Gs, \Xi, \Theta)]$, and let $t_e\in \G$ be the (non-trivial) stable letter in $\G$ associated to the edge $e$. Let $a_0$ and $a_1$ be the generators of $\Gs(d_0(e))$ and $\Gs(d_1(e))$, respectively, and define the function $\varepsilon_{p, d_1(e)}$ on the tree $\Upsilon-e$. By property (a), the graph of groups $(\Gs\at{\Upsilon},\Upsilon)$ is isocratic and every prime $p$ lies in the isocracy locus, so (\ref{eq:epsilonFormula}) yields $\varepsilon_{p, d_1(e)}(d_0(e))= \nu_p(\yi_1(e))-\nu_p(\yi_0(e))$ and
    \[|\phi_\Delta(a_0)|=|\phi_\Delta(a_1)|-\varepsilon_{p, d_1(e)}(d_1(e))=|\phi_\Delta(a_1)| + \nu_p(\yi_0(e)) -\nu_p(\yi_1(e))\]
    must hold in $C_p{^l}$, where the first equality derives from (\ref{Eq::Order}). Thus
    \[
    |\phi_\Delta(a_0)^{\yi_0(e)}|=|\phi_\Delta(a_0)|-\nu_p(\yi_0(e))
    =|\phi_\Delta(a_1)|-\nu_p(\yi_1(e))
    =|\phi_\Delta(a_1)^{\yi_1(e)}|
    \]
    in $C_p^l$, and there exists an automorphism $\alpha_e\in \operatorname{Aut}(C_p^l)$ such that $\alpha_e(\phi(a_0)^{\pm\yi_0(e)})=\phi(a_1)^{\pm\yi_1(e)}$, where the signs of $\yi_0(e)$ and $\yi_1(e)$ are chosen to agree with the signs of the canonical inclusions of $\Gs(e)$ into $\Gs(d_0(e))$ and $\Gs(d_1(e)$, respectively. Thus one may check that $\phi\colon\Delta\cup{t} \to C_{p^l} \rtimes \Aut(C_{p^l})$ generated by
    \[
    \phi(g)=
    \begin{cases}
       \phi_\Delta(g) &\quad\text{if }g\in \Delta\\
       \alpha_e &\quad\text{if }g=t_e \\ 
     \end{cases}\]
    extends to a homomorphism with properties (i) and (ii). This completes the induction and demonstrates that a morphism $\phi$ with properties (i) and (ii) will exist for any generalised Baumslag--Solitar group. The result then follows by classification of procyclic groups \cite[Propositions 2.3.8 and 2.7.1]{Bible}.
\end{proof}
The converse of Lemma~\ref{Lem::GoodIfFullTop} holds in the case where $\Xi$ is not a cycle.
\begin{lemma}\label{lem:GoodonlyifFullTop}
    Let $\G = \PA$ be a generalised Baumslag--Solitar group over a reduced graph of groups $\ggx$ such that the graph $\Xi$ is not a cycle. Then $\G$ has separable cohomology only if $\G$ induces the full profinite topology on all vertex groups $\mathcal{G}(v)$ for $v\in V(\Xi)$.
\end{lemma}
\begin{proof} 
    Assume that there exists $v\in V(\Xi)$ such that $\G$ does not induce the full profinite topology on $\mathcal{G}(v)$. If $\Xi$ is a tree then $\G$ induces the full topology on all of its vertex groups by Lemma~\ref{lem:aIffb}, so we may assume that this is not the case. If $\Gh$ has torsion then $\cd(\Gh) = \infty$ and $\G$ fails to have separable cohomology. Hence, we may assume that for all vertices $v\in V(\Xi)$, the closure $\Gsc(v)$ is torsion-free and $\G$ induces the same topology on all of its vertex groups (cf. Lemma~\ref{Lem::DifferenceIsTorsion}). It follows then from classification of procyclic groups \cite[Proposition 2.3.8 and Proposition 2.7.1]{Bible} that there exists some collection of primes $\mathcal{S}$ such that for all $v\in V(\Xi)$ we have that $\overline{\mathcal{G}}(v)\cong\prod_{p\in\mathcal{S}}\Z_p$, and that there exists some prime $q\notin\mathcal{S}$. We shall demonstrate that there is a $\G$-module $M$ whose order is a power of $q$, such that $H^2(\G,M) = 0$ but $H^2(\G,M) \neq 0$.

    First, we observe that $H^2(\Gh, M) = 0$ for any $\G$-module whose order is a power of $q$. Indeed, Proposition~\ref{Prop::Mayer-Vietoris} yields a long exact sequence, of which a portion reads
    \[
        \ldots\to\bigoplus_{e \in E(\Xi)} H^1(\Gsc(e),M)  \rightarrow  H^2(\Gh,M) \rightarrow \bigoplus_{v \in V(\Xi)} H^2(\Gsc(v),M)\to\ldots
    \]
    and we note that $H^1(\Gsc(e), M) = H^2(\Gsc(v), M) = 0$ for all vertices $v\in V(\Xi)$ and edges $e\in E(\Xi)$, via Lemma~\ref{Lem::CoprimeCohomologyVanishes} and the assumption that $q \notin \Sc$. Hence also $H^2(\Gh,M)\cong 0$ by exactness. 

    On the other hand, we shall construct a $\G$-module $M$ for which $H^2(\G, M) \neq 0$. We split the proof into two cases based on the complexity of $\Xi$. For the first case, assume that $\Xi$ contains a single cycle $\Upsilon$, and let $v_i$ be a leaf of $\Xi$ with $e_i$ the unique edge incident on $v_i$. Assume further that $e_i$ is oriented towards $v_i$. We argue here in a similar way to Lemma~\ref{Lem::FunnyModule}. Let $M$ be the $\mathbb{F}_q$-vector space $M = \mathbb{F}_q^{\yi_1(e_i)}$ and fix a basis $B$ for $M$. Let $\alpha\in \operatorname{Aut}(M)$ be the automorphism of $M$ induced by a cyclic permutation of $B$, and consider the homomorphism $\phi\colon\G\rightarrow\operatorname{Aut}(M)$ generated by
    \[
    \phi(a_j)=\begin{cases}
        \alpha, & v_j=v_i \\
        0 & \text{else},
    \end{cases}\]
    and $\phi(t) = 0$, where we use the generators of $\G$ as in (\ref{Eq::CanonicalPresentation}) and with vertices $v\in V(\Xi-\Upsilon)$ labelled in any way starting from $v_{s+1}$. We endow $M$ with a $\G$-module structure with action via $\phi$. Then all edge groups act trivially on $M$, so 
    \[
        \bigoplus_{e\in E(\Xi)}H^1(\Gs(e), M) \cong M^{|E(\Xi)|}
    \]
    using that if $M$ is a trivial $\Z$-module then $H^1(\Z, M)\cong M$. All vertex groups other than $\Gs(v_i)$ also act trivially on $M$, so it only remains to consider $H^1(\Gs(v_i), M)\cong M/(a_i-1)M$. In this case, $(a_i-1)M$ is a non-trivial subspace (having assumed $\yi_1(e_i)\neq 1$ as $\ggx$ is reduced), so $\dim_{\mathbb{F}_q}(H^1(\Gs(v_i), M)) < \dim_{\mathbb{F}_q}(M)$ and it follows that
    \begin{equation}\label{Eq::Sizes}
        \dim_{\mathbb{F}_q}\left(\bigoplus_{v\in V(\Xi)}H^1(\Gs(v), M)\right)<\dim_{\mathbb{F}_q} \left(\bigoplus_{e\in E(\Xi)}H^1(\Gs(e), M)\right)
    \end{equation}
    using that for a graph $\Xi$ containing exactly one cycle, the equation $|E(\Xi)|=|V(\Xi)|$ holds. Again,  Lemma~\ref{Prop::Mayer-Vietoris} yields the following section of a long exact sequence in abstract cohomology
    \[
        {\ldots\to\bigoplus_{v \in V(\Xi)} H^1(\mathcal{G}(v),M)} \xrightarrow{\hbar} {\bigoplus_{e \in E(\Xi)} H^1(\mathcal{G}(e),M)} \to
        H^2(\G,M)\to\ldots,
    \]
    so $\hbar$ cannot be a surjection by (\ref{Eq::Sizes}). It follows by exactness that $H^2(\G, M) \neq 0$ and $\G$ cannot not have separable cohomology.
    
    For the second case, assume that $\Xi$ contains more than one cylce, or equivalently, that $\operatorname{rk}_\Z H_1(\Xi) \geq 2$ holds. Let $M = \mathbb{F}_q$ with trivial $\G$-module structure. We have the same section of the long exact sequence in abstract cohomology, which simplifies in this case to
    \[
        \ldots\to {M^{|V(\Xi)|}} \xrightarrow{\hbar} {M^{|E(\Xi)|}} \to H^2(\G,M)\to\ldots
    \]
    using the fact that every edge and vertex group of the graph of groups description of $\G$ over $\Xi$ is a copy of $\Z$, and that the module $M$ is trivial. By assumption, $\operatorname{rk}_\Z H_1(\Xi) \geq 2$, so  $|V(\Xi)|<|E(\Xi)|$ and $\hbar$ cannot be surjective once again. Thus $H^2(\G, M) \neq 0$ again by exactness, and $\G$ cannot not have separable cohomology.
\end{proof}
\MTB*
\begin{proof} 
    The equivalence of (2a) and (2b) has been established in Lemma~\ref{lem:aIffb}. We already know from Theorem~\ref{Thm::MTALoopCase} that $\G$ has separable cohomology whenever (1) holds. On the other hand, (2b) yields the same conclusion via Lemma~\ref{Lem::GoodIfFullTop}.
    
    Conversely, suppose that $\G$ has separable cohomology. If the graph $\Xi$ is a cycle, then (1) holds via Theorem~\ref{Thm::MTA}. Otherwise, Lemma~\ref{lem:GoodonlyifFullTop} implies that (2b) holds.
\end{proof}
\printbibliography
\end{document}